\documentclass[11pt]{amsart}
\usepackage[applemac]{inputenc}

\usepackage{fullpage}
\usepackage{
 ulem, 
 amsfonts}
\usepackage{amsmath,esint, amssymb, amsthm}

\usepackage[colorlinks=true, pdfstartview=FitV, linkcolor=blue, citecolor=blue,
 urlcolor=blue]{hyperref}
\usepackage{color}

\date{}
\definecolor{sah}{rgb}{0.66,0.33, 0.04}
\definecolor{adel4}{cmyk}{1,0,0,0}
\definecolor{adel3}{rgb}{0.66,0.33, 0.04}
\definecolor{adel1}{cmyk}{0,0.20,1,0}
\definecolor{adel2}{cmyk}{0,0.40,1,0.30}
\definecolor{adel0}{rgb}{0.99,0.60, 0.30}
\definecolor{trut}{rgb}{0.99,0.80, 0.00}
\definecolor{trus}{rgb}{0.00, 0.50, 0.00}
 \definecolor{trust}{rgb}{0.99, 0.99, 0.80}
\definecolor{MaCouleur}{rgb}{0,0.9,0.3}
\def\virgp{\raise 2pt\hbox{,}}
\def\({\left(}
\def\){\right)}

\def\<{\langle}
\def\>{\rangle}

\theoremstyle{plain}

\newtheorem{Theo}{Theorem}[section]
 \newtheorem{lem}[Theo]{Lemma}

 \newtheorem{prop}[Theo]{Proposition}
 \newtheorem{Coro}[Theo]{Corollary}
 \newtheorem{defi}[Theo]{ Definition}
 
 \newtheorem{rema}[Theo]{Remark}

\newcommand{\supp}{\textnormal{supp }}

\newtheorem*{gracies}{Acknowledgements}

\newcommand{\CC}{\mathbb{C}}  
\newcommand{\NN}{\mathbb{N}} 

\newcommand{\Div}{{\textnormal{div}}}
\newcommand{\rot}{{\textnormal{curl}}}
\newcommand{\RR}{{\mathbb R}}
\newcommand{\EE}{{\varepsilon}} 
\usepackage[textmath,displaymath,floats,graphics,delayed]{preview}
 \title[]{On the Yudovich solutions for the ideal MHD equations}


\author[T. Hmidi]{Taoufik Hmidi}
\address{IRMAR, Universit\'e de Rennes 1\\ Campus de
Beaulieu\\ 35~042 Rennes cedex\\ France}
\email{thmidi@univ-rennes1.fr}

\begin{document}

\maketitle

\begin{abstract}
In this paper, we address the problem of weak solutions of  Yudovich type for the inviscid MHD equations in two dimensions.  The local-in-time existence and uniqueness of these  solutions  sound to be hard to achieve  due  to some  terms involving  Riesz transforms in the vorticity-current formulation. We shall prove  that the vortex patches with smooth boundary  offer   a suitable class of initial data for which the problem can be solved. However  this is only done under a geometric constraint  by assuming  the  boundary of the initial vorticity to be  frozen  in a magnetic field line. 

We shall also discuss   the  stationary patches for the incompressible Euler system $(E)$ and the MHD system.  For example, we prove that a stationary simply connected  patch with rectifiable boundary for  the system $(E)$ is necessarily the  characteristic function of a disc. 
\end{abstract}
\begin{quote}
\footnotesize\tableofcontents
\end{quote}

\section{Introduction}
In this paper we shall consider a fluid which is electrically conducting and moves through a  prevalent magnetic fields. The interaction between the motion and the magnetic  fields are governed by the coupling between Navier-Stokes system  and Maxwell's equations in the magnetohydrodynamics approximation. The basic equations  of hydromagnetics are given by,
\begin{equation}
\label{MHD0}
 \left\{ 
\begin{array}{ll} 
\partial_t v+v\cdot\nabla v-\nu\Delta v+\nabla p=b\cdot\nabla b,\qquad t>0,\, x\in \mathbb R^d,\, d\in\{2,3\},  \\
\partial_t b+v\cdot\nabla b-\mu\Delta b=b\cdot\nabla v,\\
\textnormal{div }v=\textnormal{div }b=0,\\
v_{\mid t=0}= v_{0},b_{\mid t=0}= b_{0},
\end{array} \right.    
     \end{equation}
 where $v$ denotes the velocity of the fluid particles  and $b$ the magnetic field which are both assumed to be solenoidal. The pressure $p$ is a scalar function that can be recovered from the velocity and the magnetic field by inverting an elliptic equation. The parameters $\nu,\mu\geq0$ are called the viscosity and  the resistivity, respectively.
 
 \quad  The major use of MHD is  in liquid metal and  plasma physics and the  derivation of the governing equations   can be done by using Maxwell's equations where we neglect the displacement currents,
 $$
 \textnormal{div }b=0,\quad\textnormal{curl }b=4\pi  J\quad\hbox{and}\quad\textnormal{curl } E=-\mu\partial_t b
 $$
 with $E$  the electric field, $J$ the current density and $\mu$ the magnetic permeability. To complete the fields equations  we need   an equation for the current density $J$ which requires some assumption on the nature of the fluid. This is described by Ohm's law
 $$
 J=\sigma\big(E+\mu\, v\times b\big).
 $$
 The combination of the preceding equations will lead to the second equation of the magnetohydrodynamics system \eqref{MHD0}.  When the conducting fluid is in motion currents are induced  and  the magnetic field will in turn act on the fluid  according to  Lorenz force $\mathcal{L}$
 $$
 \mathcal{L}\triangleq \mu J\times b=\frac{\mu}{4\pi}\textnormal{curl }b\times b,\quad \textnormal{curl }b\times b=b\cdot\nabla b-\frac12\nabla|b|^2.
 $$
 Observe that in the foregoing formula  Lorentz force is decomposed into two parts: the first one which appears in the  first equation of \eqref{MHD0} is called a {{curvature force}} and acting toward the center of curvature of the field lines. The second one is a magnetic pressure which acts perpendicular to the magnetic fields and it is implicitly contained in the pressure term $p$. 
 For a general review about the derivation of the MHD equations and some dynamical aspects of the interaction between the magnetic fields and the velocity we can consult the references \cite{Bis,Chan,Dav,Duvaut}. 
 
\quad The theoretical study of  the  MHD system has started with the pioneering work of Alfv\'en \cite{Alf} who  was  the first to describe the generation of electromagnetic-hydrodynamic  waves by conducting liquid  using the MHD equations.  From mathematical point of view a lot of  progress has be done from that time. For example,   the local well-posedness  theory which is a  central subject in modern PDEs is carried out   in various  classical function spaces, see for instance  \cite{Caf,Casella,Chen1,Fef, Jiu,Kozono,Zhang,Mia2,Tem,Wu} and the references therein.   However the global existence of such solutions is an  open problem  except in two dimensions  with the full dissipation $\nu,\mu>0.$ \

\quad From now onwards we shall   focus  only  on ideal  MHD fluid corresponding to   $\nu=\mu=0$ and therefore  the equations become  \begin{equation}
\label{MHD}
 \left\{ 
\begin{array}{ll} 
\partial_t v+v\cdot\nabla v+\nabla p=b\cdot\nabla b,\qquad x\in \mathbb R^d,\, t>0, \\
\partial_t b+v\cdot\nabla b=b\cdot\nabla v\\
\textnormal{div }v=\textnormal{div }b=0,\\
v_{\mid t=0}= v_{0},b_{\mid t=0}= b_{0},
\end{array} \right.    
     \end{equation}
     \quad One of the most important consequence of the second equation of \eqref{MHD} and known in the literature by {\it Alfv\'en's theorem} is the freezing of the magnetic field lines into the fluid; this means  that the magnetic  lines follow the motion of the fluid particles. We note that the ideal MHD is quite successful model  for large-scale plasma physics and can be illustrated in various phenomena   in Earth's magnetosphere and on the sun like the  {\it sunspots}.  As we shall see  the frozen-in magnetic fields  will be of crucial importance in our study of weak solutions of Yudovich type in the two dimensional space.  
    
\quad It is in some extent true   that the system \eqref{MHD} is at a formal level a perturbation of the incompressible Euler equations and therefore it is legitimate to see whether the known results for Euler equations work for the MHD system as well. For example, it is proved in  \cite{Sch, Secc} that  the commutator theory developed by  Kato and Ponce in  \cite{Kato} can be successfully implemented leading to  the local well-posedness for \eqref{MHD}  when the initial data $v_0, b_0$ belong to the sub-critical Sobolev space $H^s, s>\frac{d}{2}+1$ and the maximal solution  satisfies  $v, b\in \mathcal{C}([0,T^\star); H^s)$.  Whether or not the \mbox{lifespan  $T^\star$} is finite is an outstanding open problem even in two dimensions. However it is well-known that for planar motion and in the absence of the magnetic \mbox{field $b_0=0$,} classical  solutions are global in time since  the \mbox{vorticity  $\omega\triangleq \partial_1 v^2-\partial_2 v^1$} is transported by the flow, namely we have
 \begin{equation}\label{tran30}
 \partial_t\omega+v\cdot\nabla \omega=0.
 \end{equation}
 This shows that Euler equations have a Hamiltonian structure  and  gives in turn an infinite family of conservation laws such as $\|\omega(t)\|_{L^p}=\|\omega_0\|_{L^p}$ for any $p\in [1,\infty].$ These global  a priori estimates  allow \mbox{Yudovich \cite{Y1}} to  relax  the classical regularity and establish the global existence and uniqueness only with $\omega_0\in L^1\cap L^\infty$. Unfortunately, as we shall see the structure of the vorticity is instantaneously altered for the model \eqref{MHD} due  to the effects of the magnetic fields. This fact will be a source of   at least two main difficulties. The first one is connected to the global existence of classical solutions where no strong global  a priori estimates are known till now. The second one concerns Yudovich solutions whose construction is not at all clear even for short time. This  can be clarified  through the equations governing     the vorticity and the current density $j=\partial_1 b^2-\partial_2 b^1$, 
 \begin{equation}
\label{E0}
 \left\{ 
\begin{array}{ll} 
\partial_t \omega+v\cdot\nabla\omega=b\cdot\nabla j\\
\partial_t j+v\cdot\nabla j=b\cdot\nabla\omega+2\partial_1b\cdot\nabla v^2-2\partial_2b\cdot\nabla v^1.
\end{array} \right.    
     \end{equation} 
     We observe that the magnetic field contributes in the last nonlinear part of the second equation with the  quadratic term
     \begin{equation}\label{source}
\mathcal{H}(v,b)\triangleq  2\partial_1b\cdot\nabla v^2-2\partial_2b\cdot\nabla v^1
\end{equation}
 which can be  described as a linear superposition of the quantities
     $
     \mathcal{R}_{ik}\omega\mathcal{R}_{lm}j,
     $
     where $\mathcal{R}_{ik}=\partial_{i}\partial_k\Delta^{-1} $ is the iterated Riesz transform.
     The main step when we wish to deal with Yudovich solutions is to be able to propagate the $L^p\cap L^\infty$ bound of the vorticity for some finite value of $p$. This problem is not trivial due to two effects. The first one is  the lack of continuity of Riesz transform on the bounded functions; and the second one concerns the nonlinear structure of the term $\mathcal{H}(v,b).$ We point out that even for finite value of $p$  no global a priori estimates are known in the literature and their persistence  requires  the velocity to be in the Lipschitz class.  
     
 \quad    One of the main scope of this paper is to be able to construct local unique solutions for a sub-class of Yudovich data. In broad terms, we shall see  that the vortex patches offer a suitable class of initial data for which the construction of Yudovich solutions is possible.
      But before stating our result let us briefly discuss what is known for Euler equations with this special initial data. First, we say that a vorticity $\omega_0$ is a patch if it  is constant inside a bounded set $\Omega$ and vanishes outside, namely and by normalization we can take  $\omega_0=\chi_{\Omega}$. It is clear from the transport equation \eqref{tran30} that this structure is not altered through the time and the  vorticity remains always a patch. This means that   for any positive time $\omega(t)=\chi_{\Omega_t},$ with $\Omega_t\triangleq \psi(t,\Omega)$ being the image of $\Omega$ by the flow.  A connected problem that was raised first in the numerical studies and leading later to a nice  theoretical  achievement 
      was to understand whether or not  the boundary develops  finite-time singularities.  In \cite{che1}, Chemin proved that when we start  with a smooth boundary, say    $\partial \Omega$ belongs to the H\"{o}lderian class $C^{1+\varepsilon}, 0<\varepsilon<1,$ then for any time $t$ the boundary $\partial \Omega_t$ remains in the same class. The basic idea of Chemin is that only the co-normal regularity $\partial_X\omega$ of the vorticity contributes for the Lipschitz norm of the velocity. The choice of the vector fields $(X_t)$ can be done in such a way that it should be  tangential  to $\partial\Omega_t$ for any positive time. This is satisfied when it is  transported by the flow, that is,
      \begin{equation}\label{tr71}
      \partial_t X+v\cdot\nabla X=X\cdot\nabla v.
      \end{equation}  
   One of the main feature of these vector fields is their commutation  with the transport \mbox{operator $\partial_t+v\cdot\nabla $,} which leads to the important equation
   \begin{equation}\label{master}
  (\partial_t +v\cdot\nabla) \partial_X\omega=0. 
   \end{equation}
   This means that the co-normal regularity of the vorticity is also transported by the flow and this is the crucial tool in the framework of the vortex patches.  
   
  \quad Our main concern here  is to valid similar results for the MHD equations and as we shall see the situation is slightly more complex. The presence of the magnetic field will contribute with two opposite effects. First, it will destroy the structure of the vortex patches and introduce nonlocal singular operators of Calder\'on-Zygmund type. Second, the fact that the magnetic field is transported by the flow- it is a push-forward vector field- will be of great importance especially for measuring  the co-normal regularity of the vorticity.

\quad   Before stating our contribution in  this subject we shall discuss a little bit an intermediary problem concerning the stationary patches. This consists in  finding  simply connected bounded domains $\Omega$ and $D$ such that $\omega(t)=\chi_{\Omega}$ and $j(t)=\chi_D$ define a solution for the vorticity-current formulation \eqref{E0}.  We can analyze  the same problem  for the $2d$ incompressible Euler equations. The only example that we know for this latter model  is the Rankine vortices corresponding to the domains with circular shape.  We will show that these are in fact  the only stationary patches.  This expected but non trivial result can be obtained from Fraenkel's theorem on potential theory as we shall see later in Section \ref{stat} and whose proof is based on many tools of elliptic equations.  For the MHD system, we will conduct the same study with the same tools  and our results can be summarized in the following theorem. 

  \begin{Theo}\label{The44}
  The following assertions  hold true.
    \vspace{0,2cm}
    
  ${\bf{I)}}$ Let $\Omega$ be a simply connected domain with rectifiable Jordan boundary.  Then $\chi_\Omega$ is a stationary  patch for the $2d$ Euler equations if and only if  $\Omega$ is  a  ball.
  \vspace{0,2cm}
  
 ${\bf{II)}}$  
Let $D$ and $\Omega$  be two bounded domains   and $\omega_0={\chi}_{\Omega}$,   $j_0={\chi}_{D}.$
\begin{enumerate}
\item If $D=\Omega$ then $(\omega_0, j_0)$ is a stationary solution for the MHD system \eqref{MHD}.
\item If the boundaries $\partial D$ and $ \partial\Omega$ are disjoint and  rectifiable   then $(\omega_0, j_0)$ is a stationary solution for the system \eqref{MHD} if and only if  $\Omega$ and $D$ are concentric balls.

\end{enumerate}
\end{Theo}
Few remarks are in order.
\begin{rema}
In the statement ${\bf{II)}}-(1)$ of the preceding theorem, there are no constraints on the domain $\Omega$. This is due to the special structure of the inviscid MHD equations: if we take $b_0=v_0$ then we can readily check  that this corresponds to a stationary solution for \eqref{MHD} without pressure.  This illustrates one of  the deepest  and rigid geometric structure of the magnetic field which forces here the motion to be independent in time. 

\end{rema}
\begin{rema}
The stationary patches of Euler equations appear as a special case  of rotating patches whose study were done in a series of papers such as \cite{B,HMV}.

\end{rema}
Some additional  remarks and comments will be raised in Section \ref{stat}. Now we shall come back to  the consideration of Yudovich solutions in the framework of  vortex patches and  we shall formulate a general  statement   later in Theorem \ref{thm2}.
     
     \begin{Theo}\label{thm1}
     Let  $\Omega$ be a simply connected domain of class $W^{2,\infty}$ and $\omega_0=\chi_{\Omega}$. Let $b_0=\nabla^\perp\varphi_0$ be a divergence-free magnetic field  such that its current density $j_0$ belongs to $L^1\cap W^{1,p}$, with $2<p<\infty$.
     Assume that:
     \begin{enumerate} 
     \item Compatibility assumption:
     $$
     b_0\cdot n=0\quad \hbox{on}\quad \partial \Omega,
     $$
     where $n$ is a normal vector to the boundary $\partial\Omega$.
     \item There exist two constants $\delta,\eta>0$ such that
     \begin{equation}\label{Equiv1}
    \forall \, x\in \RR^2,\quad |\varphi_0(x)-\lambda|< \eta \Longrightarrow |b_0(x)|> \delta
     \end{equation}
     where $\lambda$ is the value of $\varphi_0$ on the boundary $\partial\Omega$.
     \end{enumerate}
     Then there exists $T>0$ and a unique solution $(v,b)$ for the system \eqref{MHD}  with 
$$\omega, j\in L^\infty([0,T];L^1\cap L^\infty)\quad\hbox{and}\quad v,b\in L^\infty([0,T]; \textnormal{Lip}).
$$
Moreover, for any $t\in [0,T],$ the boundary of $\psi(t,\Omega)$ belongs to $W^{2,\infty}$.
     \end{Theo}
       \vspace{0,2cm}
       
     Before giving some details about the proof, we shall give few  remarks.
      \begin{rema}
      It is worth noting that the compatibility assumption is not only  restrictive to the vortex patch problem but appears in the  {\it current vortex sheets} called also in the literature by  the  {\it{ MHD tangential discontinuity}}. The construction of local in time piecewise smooth solutions  apart from  a smooth  hypersurface $\Gamma_t$ is known     provided that the magnetic field $b_0$ is tangential to $\Gamma_0$ and a stability condition is satisfied at each point of the initial discontinuity. For more details see \cite{Coul, Mor}  and the references therein.
       \end{rema}
     \begin{rema}
     The existence of $\varphi_0$ in the foregoing theorem  follows from the incompressibility of $b_0$ which is a Hamiltonian vector field. Moreover,  since $b_0$ is co-normal  to the  connected curve $\partial\Omega$ according to the assumption $(1)$, then necessarily  this curve must be  a level set of $\varphi_0$ and this justifies the existence of $\lambda$ in the assumption $(2)$. For more details  see Proposition $\ref{prop754}. $
     
     \end{rema}
     \begin{rema}
     The  compatibility assumption $(1)$  imparts to the magnetic field some rigidity: it must be singular for at least one point  inside the domain $\Omega.$ This follows easily from the fact that the \mbox{Hamiltonian $\varphi_0$} is constant on the boundary and thus it has a critical point in $\Omega.$
     \end{rema}
      \begin{rema}
     The condition $\eqref{Equiv1}$ implies in particular that  the extrema of the Hamiltonian function $\varphi_0$  should not be located on the regular level surface energy  containing the curve $\partial\Omega.$  This means somehow that the magnetic field must be regular close to this level set. This assumption is very strong and unfortunately it does not  allow to reach Chemin's result for the Euler case  corresponding to $b_0=0.$ It seems that the restriction described  by the compatibility assumption $(1)$ is relevant and essential in our analysis since  it induces deep algebraic  structure; we need that any co-normal vector field to the initial patch must commute with the initial magnetic field.  However we can hope to dispense with  the non degeneracy assumption  of the magnetic field around the boundary which sounds  to be  a technical artifact.
     \end{rema}
      \begin{rema}
      As we have already seen, Chemin proved in \cite{che1} the global persistence of the $C^{1+\varepsilon}, 0<\varepsilon<1$ boundary  regularity  for the two dimensional Euler equations. But in our main result we require more:   the boundary should be at least  in the class $W^{2,\infty}$. This is due to the following  technical fact: the space $C^{\varepsilon-1}\cap L^\infty$ used naturally to measure  the co-normal regularity  is not an algebra and to overcome this difficulty we should work with positive index spaces.
       \end{rema}
      \begin{rema}
      As we shall see next  in Lemma $\ref{lemex}$, the assumptions $(1)$ and $(2)$ of Theorem $\ref{thm1}$  are not empty.
       \end{rema}

     \vspace{0,2cm}
     
     {\bf Outline of the proof.}
     The proof uses the standard formalism of vortex patches developed by Chemin in \cite{Chez, che1} for incompressible Euler equations. As we have already seen, one of the main feature of Euler equations is the commutation of the push-forward vector fields given by \eqref{tr71} with the transport operator leading to the master equation \eqref{master}. This algebraic property is instantaneously destroyed by the magnetic field which contributes with additional terms as the following equations show
      \begin{equation}
\label{E0X}
 \left\{ 
\begin{array}{ll} 
\partial_t\partial_X \omega+v\cdot\nabla \partial_X\omega=b\cdot\nabla \partial_X j+\partial_{\partial_X b-\partial_b X} j\\
\partial_t \partial_X j+v\cdot\nabla \partial_X j=b\cdot\nabla\partial_X\omega+\partial_{\partial_X b-\partial_b X}\omega +\partial_X\{2\partial_1b\cdot\nabla v^2-2\partial_2b\cdot\nabla v^1\}.
\end{array} \right.    
     \end{equation} 
Thus  and in order to get similar equations to \eqref{E0}  we should at this stage kill the terms involving the vector field $\partial_X b-\partial_b X$. Therefore we shall assume that the vector fields $X$ and $b$  commute initially and this algebraic property is not altered through the time. For the sake of simplicity we can make the choice $X=b$ and   this algebraic constraint will lead in the special case of the vortex patch to the geometric constraint described by the compatibility assumption.  It is worth noting that the main obstacle to reach the regularity $C^{1+\varepsilon}$  for the boundary is the estimate of the last term of the system \eqref{E0X} in the space $C^{\varepsilon-1}$  and  it is not at all clear how to proceed since $C^{\varepsilon-1}\cap L^\infty$ is not an algebra. Besides the geometric condition stated in the compatibility assumption   will force the Hamiltonian magnetic field to be degenerate at least at some points inside the domain $\Omega$ and subsequently  we shall  get from the vortex patch formalism some useful information only  far from this singular set. In this region, it is not clear how to construct a non degenerate vector field which  commutes with the magnetic field. To circumvent this difficulty we use  that the initial data are smooth wherever the magnetic field is degenerating combined with the finite speed of propagation of the transport operator. So  for a short time we expect    the influence of the singular parts to  be localized  close  to the image by the flow of the initial one. This fact is not  quite trivial due to the nonlocal property of Riesz transforms in \eqref{source} and thus  some elaborated analysis are required. Especially, the truncation of  the solutions far from  the singular set should be done in  a special way by cutting along the streamlines of the magnetic field.  We emphasize that in this step we use an  algebraic identity for the last term  of \eqref{E0}, see \eqref{Iden1}, combined with  Calder\'on commutator type estimates.

\quad The paper is structured as follows. In Section $2$ we  recall some classical spaces frequently  used in the vortex patch problem. We end this section with some results on the persistence regularity for various transport models. In Section $3$ we shall review some basic results on the algebra  vector fields . In Section $4$ we  detail some weak estimates for both the vorticity and the current density. In Section $5$ we shall be concerned with the stationary patches and we plan to give the proof of Theorem \ref{The44}. Some general facts on conformal mapping and rectifiable boundaries will be also  discussed. Section $6$ will be devoted to the  proof of Theorem \ref{thm1} and its extension to generalized vortex patches. Finally, we shall close this paper by some commutator estimates. 
\section{Basic tools}\label{Sec2}
In this section we shall introduce some function spaces and investigate  some of their elementary properties. We will also recall few basic results concerning some transport equations. 
First we need to fix a piece of notation  that will be frequently used along this paper.
\begin{itemize}
\item For $p\in[1,\infty]$, the space $L^p$ denotes the usual Lebesgue space.
\item We denote by $C$ any positive constant that may change from line to line and by  $C_{0}$ a real positive constant depending on the size of the initial data.
\item  For any positive real numbers $A$ and $B$, the notation  $A\lesssim B$ means that there exists a positive constant $C$ independent of $A$ and $B$  such that $A\leqslant CB.$
\item For any two sets $E,F\subset \RR^2$ and $x\in\RR^2$, we define
$$
d(x,E)\triangleq\inf\{|x-y|; y\in E\};\quad \textnormal{dist}(E,F)\triangleq\inf\{d(x,F); x\in E\}.
$$
\item For a subset $A\subset\RR^2$, we denote by $\chi_A$ the characteristic function of $A$ which is defined by
 \begin{equation*}
\chi_A(x)=
 \left\{ 
\begin{array}{ll} 
1,\quad \hbox{if}\quad x\in A,\\
0\quad \hbox{if}\quad x\notin A
\end{array} \right.    
     \end{equation*} 
\end{itemize}
\subsection{Function spaces}

In what follows we intend to recall  the definition of H\"{o}lder spaces $C^{\alpha}$ and Sobolev spaces of type $W^{1,p}$. 
Let  $\alpha\in]0,1[$, we denote by $C^\alpha$ the set of continuous functions $u:\RR^d\to\RR$ such that
$$
\|u\|_{C^\alpha}=\|u\|_{L^\infty}+\sup_{x\neq y}\frac{|u(x)-u(y)|}{|x-y|^\alpha}<\infty.
$$
The Lipschitz class denoted by $\textnormal{Lip}$  corresponds to the borderline  case $\alpha=1,$ 
$$
\|u\|_{\textnormal{Lip}}=\|u\|_{L^\infty}+\sup_{x\neq y}\frac{|u(x)-u(y)|}{|x-y|}<\infty.
$$
We shall also make use of the space $C^{1+\alpha}(\RR^d)$ which is the set of continuously differentiable \mbox{functions $u$} such that
$$
\|u\|_{C^{1+\alpha}}=\|u\|_{L^\infty}+\|\nabla u\|_{C^\alpha}<\infty.
$$
By the same way we can define  the spaces $C^{n+\alpha},$ with $n\in\NN$ and $\alpha\in]0,1[.$ 

Now we shall recall Sobolev space $W^{1,p}$ for $p\in [1,\infty]$, which is the set of the tempered distribution $u\in \mathcal{S}^\prime $ equipped with the norm
$$
\|u\|_{W^{1,p}}\triangleq \|u\|_{L^p}+\|\nabla u\|_{L^p}.
$$

Our next task is to  introduce the anisotropic Sobolev spaces, which are the analogous of the anisotropic H\"older spaces introduced by Chemin some years ago in \cite{che1}. 
\begin{defi}\label{Defin79}
Let $\varepsilon\in(0,1), p\in [1,\infty]$ and  $X:\RR^2\to\RR^2$ be a  smooth divergence-free vector field. Let $u:\RR^2\to\RR$  be a scalar  function in $L^1\cap L^\infty$.
\begin{enumerate}
\item We say that $u$ belongs to the space $C^\varepsilon_X$ if and only if
$$
\|u\|_{C^\varepsilon_X}\triangleq \|u\|_{L^1\cap L^\infty}+\|\partial_Xu\|_{ C^{\varepsilon-1}}<\infty.
$$
\item The function $u$ belongs to the space ${W}^{p}_X$ if and only if
$$
\|u\|_{{W}^{p}_X}\triangleq \|u\|_{L^1\cap L^\infty}+\|\partial_Xu\|_{ L^{p}}<\infty
$$
where  we denote by 
$$
\partial_X u=\textnormal{div}(X u).
$$
\end{enumerate}
\end{defi}
We will see later in Section \ref{vec89} some additional properties about the Lie derivative $\partial_X$. 

Now we shall introduce the notion of H\"{o}lderian singular support. 
\begin{defi}\label{singsuppp}
Let $\varepsilon\in]0,1[$ and $u:\RR^2\to\RR$. We say that $x\notin \Sigma_{\textnormal{sing}}^\EE(u)$ if there exits  a smooth function $\chi$ defined in a  neighborhood  of $x$ with $\chi(x)\neq0$ and $\chi u$ belongs to $C^\EE.$ 

The closed set $\Sigma_{\textnormal{sing}}^\EE(u)$ is called the H\"{o}lderian singular support of $u$ of index $\varepsilon$.
   \end{defi}

{\bf{Example:}} Let $\Omega$ be a Jordan domain and $u=\chi_\Omega$ be the characteristic function of $\Omega$. Then
$$
\Sigma_{\textnormal{sing}}^\EE(u)=\partial\Omega.
$$
Moreover, if the  boundary is $C^1$  then
$$
\partial_X \chi_\Omega=-(X\cdot\vec{n})d\sigma_{\partial\Omega},
$$
with $d\sigma_{\partial\Omega}$ the arc-length measure on $\partial\Omega$ and $\vec{n}$ the outward unit normal.
In the particular case where  $X$ is tangential, said also co-normal, to $\partial\Omega$ we get
$$
\partial_X \chi_\Omega=0.
$$

\vspace{0,2cm}

Now we shall briefly discuss  some elementary results on    the  Littlewood-Paley theory. First we need to recall the  following statement concerning the dyadic partition of the unity.
 
There exist two radial positive functions  $\chi\in \mathcal{D}(\RR^d)$ and  $\varphi\in\mathcal{D}(\RR^d\backslash{\{0\}})$ such that
\begin{itemize}
\item[\textnormal{i)}]
$\displaystyle{\chi(\xi)+\sum_{q\geq0}\varphi(2^{-q}\xi)=1}$;$\quad \displaystyle{\forall\,  \,  q\geq1,  \,   \supp\chi\cap \supp\varphi(2^{-q})=\varnothing}$\item[\textnormal{ii)}]
 $ \supp\varphi(2^{-j}\cdot)\cap
\supp\varphi(2^{-k}\cdot)=\varnothing,  $ if  $|j-k|\geq 2$.
\end{itemize}
For any $v\in{\mathcal S}'(\RR^d)$ we set the cut-off operators,
  $$
\Delta_{-1}v=\chi(\hbox{D})v~;\,   \forall
 q\in\NN,  \;\Delta_qv=\varphi(2^{-q}\hbox{D})v\quad\hbox{ and  }\;
 S_q=\sum_{-1\leq p\leq q-1}\Delta_{p}.$$
From \cite{b},  we split formally  the product $uv$ of two distributions into three parts,
$$
uv=T_u v+T_v u+R(u,  v),  
$$

\noindent with

$$
T_u v=\sum_{q}S_{q-1}u\Delta_q v,  \quad  R(u,  v)=\sum_{q}\Delta_qu\tilde\Delta_{q}v  \quad\hbox{and}\quad \tilde\Delta_{q}=\sum_{j=-1}^1\Delta_{q+j}.
$$

We will  make continuous use of Bernstein inequalities (see  \cite{che1} for instance).
\begin{lem}\label{lb}\;
 There exists a constant $C$ such that for $q,  k\in\NN,  $ $1\leq a\leq b$ and for  $u\in L^a(\RR^d)$,   
\begin{eqnarray*}
\sup_{|\alpha|=k}\|\partial ^{\alpha}S_{q}u\|_{L^b}&\leq& C^k\,  2^{q(k+d(\frac{1}{a}-\frac{1}{b}))}\|S_{q}u\|_{L^a},  \\
\ C^{-k}2^
{qk}\|{\Delta}_{q}u\|_{L^a}&\leq&\sup_{|\alpha|=k}\|\partial ^{\alpha}{\Delta}_{q}u\|_{L^a}\leq C^k2^{qk}\|{\Delta}_{q}u\|_{L^a}.
\end{eqnarray*}

\end{lem}
Now we shall recall the characterization of H\"{o}lder spaces in terms of the frequency cut-offs. 

For $s\in [0,\infty[\backslash \NN$, the usual norm of $C^s$ is equivalent to 
$$
\|u\|_{C^s}\approx \sup_{q\geq-1}2^{qs}\|\Delta_q u\|_{L^\infty}.
$$
Now we shall prove that the assumptions of Theorem \ref{thm1} can be satisfied by choosing suitably the magnetic vector field.
\begin{lem}\label{lemex}
Let $\Omega$ be a simply connected domain with boundary in $C^{1+\varepsilon}$ and $\EE\in]0,1[.$ Then we can find a Hamiltonian vector field $b_0$ satisfying the assumptions $(1)$ and $(2)$ of Theorem $\ref{thm1}$.
\end{lem}
\begin{proof}
       We will briefly outline the proof of this lemma. First, it is a well-known fact that when the boundary $\partial\Omega$ is at least $C^1$ then it  can be seen as a level set of a smooth function. More precisely, there exists a smooth function $f:\RR^2\to\RR_+$ with the following properties: 
     $$
     \partial\Omega=\{x\in\RR^2; f(x)=1\}; \,\Omega=f^{-1}(]1,+\infty[);
     $$
     $$
   \lim_{\|x\|\to\infty}f(x)=0\quad\hbox{and}\quad \nabla f(x)\neq 0, \forall x\in \partial\Omega.
     $$
   For $h>0$ introduce the sets
    $$\Omega_h\triangleq\{x;\; \textnormal{dist} (x,\Omega)\le h\};\quad \partial\Omega_h\triangleq\{x;\; \textnormal{dist} (x,\partial\Omega)\le h\}
    .$$ Then for $\eta>0$ sufficiently small, there exists $h>0$ such that 
    \begin{equation}\label{as1}
   \forall x\in \RR^2,\quad |f(x)-1|\le\eta  \Longrightarrow x\in \partial\Omega_h . 
    \end{equation}
     Now let $\chi:\RR^2\to [0,1]$ be a smooth  compactly supported function  such that $\chi(x)=1,\forall x\in \Omega_h.$ Set
    $$\varphi_0(x)=\chi(x) f(x)\quad\hbox{and}\quad b_0=\nabla^\perp \varphi_0.
    $$ Then $b_0$ satisfies the assumptions $(1)$ and $(2)$ of Theorem \ref{thm1}. Indeed, the first assumption is easy to check. As to the second one, using \eqref{as1} we easily  obtain
    $$
    \{|\chi f-1|\le\eta\} \subset \partial\Omega_h.
    $$
     Moreover, it is clear that for $x\in \Omega_h,$ 
   $$
   b_0(x)=\chi(x)\nabla^\perp f(x)=\nabla^\perp f(x).
   $$
   Since $\partial\Omega$ is a regular energy curve, then we can choose $h>0$ small enough such that, for \mbox{some $\delta>0,$}
   $$
  \forall x\in \partial\Omega_h;\quad\, |b_0(x)|\geq \delta.
   $$
   This concludes the proof of the lemma.
     \end{proof}

\subsection{Transport equations}
We intend to discuss some basic results about the persistence regularity for some  transport equations. The first one  is very classical  and whose proof can 
 found in \cite{che1} for instance. 
\begin{prop}\label{l555}
Let $v$ be a divergence-free vector field and  $F$ be a smooth function. Let $f$ be a solution of the transport equation
$$
\partial_{t} f+v\cdot\nabla f=F.
$$
Then the following estimates hold true.
\begin{enumerate}
\item $L^p-$estimates: Let $p\in [1,\infty]$ then for any $t\geq 0$
$$
\|f(t)\|_{L^p}\le\|f(0)\|_{L^p}+\int_{0}^t\|F(\tau)\|_{L^p} d\tau.
$$
\item H\"{o}lder estimates: For $\EE\in ]-1,1[$ we get
$$
\|f(t)\|_{{C^{\EE}}}\leq C e^{CV(t)}\Big(\|f(0)\|_{{C^{\EE}}}+\int_{0}^te^{-CV(\tau)}
\|F(\tau)\|_{{C^{\EE}}}d\tau\Big),
$$
with $C$ a constant depending only on the index regularity $\EE$ and 
$$
V(t)\triangleq \int_{0}^t\|\nabla v(\tau)\|_{L^\infty}d\tau.
$$
\end{enumerate}

\end{prop} 

Next we shall   deal with the same problem  for a coupled transport model generalizing the previous one and which appears naturally in the structure of the MHD system \eqref{MHD}.
\begin{equation}\label{coupled}
 \left\{ 
\begin{array}{ll} 
\partial_tf+v\cdot\nabla f =b\cdot\nabla g+F\\
\partial_tg+v\cdot\nabla g=b\cdot\nabla f+G
\end{array} \right.    
     \end{equation}
     where $F$ and $G$ are given and the unknowns are $f$ and $g$.
\begin{prop}\label{transport}
Let $v$ and $b$ be two divergence-free smooth vector fields and $f,g$ be two smooth solutions for \eqref{coupled}. Then the following estimates hold true.
\begin{enumerate}
\item $L^p-$estimates: For $p\in [1,\infty]$ we get
$$
\|f(t)\|_{L^p}+\|g(t)\|_{L^p}\lesssim\|f(0)\|_{L^p}+\|g(0)\|_{L^p}+\int_{0}^t\big(\|F(\tau)\|_{L^p}+\|G(\tau)\|_{L^p}\big) d\tau
$$
\item H\"{o}lder estimates: For $\EE\in ]-1,1[$ we get
$$
\|f(t)\|_{{C^{\EE}}}+\|g(t)\|_{{C^{\EE}}}\leq C e^{CV(t)}\Big(\|f(0)\|_{{C^{\EE}}}+\|g(0)\|_{{C^{\EE}}}+\int_{0}^te^{-CV(\tau)}
\big(\|F(\tau)\|_{{C^{\EE}}}+\|G(\tau)\|_{{C^{\EE}}}\big)d\tau\Big),
$$
with
$$
V(t)\triangleq \int_{0}^t\big(\|\nabla v(\tau)\|_{L^\infty}+\|\nabla b(\tau)\|_{L^\infty}\big)d\tau.
$$
\end{enumerate}
\end{prop} 

\begin{proof}
We shall 
introduce Elasser variables, see \cite{Ela}, 
$$
\Phi\triangleq f+g\quad\hbox{and}\quad \Psi\triangleq f-g.
$$
Then we can easily check that
\begin{equation*}
 \left\{ 
\begin{array}{ll} 
\partial_t\Phi+(v-b)\cdot\nabla \Phi =F+G\\
\partial_t\Psi+(v+b)\cdot\nabla \Psi=F-G
\end{array} \right.    
     \end{equation*}
     These are transport equations with  divergence-free vector fields and thus we can apply  \mbox{Proposition \ref{l555}} leading to the desired estimates.
\end{proof}

\section{Basic results on  vector fields}\label{vec89}
In this section we shall review some general results on vector fields and focus on some canonical commutation relations. Special attention will be paid to the Hamiltonian vector fields for which some nice properties are established. Most of the results that will be discussed soon are very known and for the commodity of the reader we prefer giving the proofs of some of them.
\subsection{Push-forward}
Let $X:\RR^n\to\RR^n$ be a smooth vector field and $f:\RR^n\to\RR$ be a smooth function. We denote by $\partial_X f$ the derivative of $f$ in the direction $X$, that is, 
$$
X(f)=\partial_X f=\sum_{i=1}^n X^i\partial_i f=X\cdot\nabla f.
$$  
This is the Lie derivative of the function $f$ with respect to the vector field $X,$ denoted usually by $\mathcal{L}_X f$ and in the preceding formula we adopt different notations for this object. 

For two vector fields $X,Y:\RR^n\to\RR^n$, their commutator is given by the Lie bracket $[X,Y]$ defined  in the coordinates system by
\begin{eqnarray*}
[X,Y]^i&\triangleq&\sum_{i=1}^n\big(X^j\partial_j Y^i-Y^j\partial_j X^i\big)\\
&=&\partial_{X}Y^i-\partial_Y X^i .
\end{eqnarray*} 
This can also be written in the form
\begin{equation}\label{comm1}
\partial_X\partial_Y-\partial_Y\partial_X=\partial_{\partial_X Y-\partial_y X}.
\end{equation}
We mention that when $f$ is not sufficiently smooth, for example $f\in L^\infty$ and this will be  mostly  the case in our context, and the vector field $X$ is  divergence-free  we define $\partial_X f$   in  a weak sense as follows,
$$
\partial_X f=\textnormal{div}(X f).
$$     
                                                                   
Now we intend to study some geometric and analytic properties of the {\it push-forward} of a vector field $X_0$ by the flow map associated to another time-dependent vector field $v(t)$.  First recall that the push-forward $\phi_\star X$  of a vector field $X$ by a diffeomorphism $\phi$ of $\RR^d$ is  given by 
$$
(\phi_\star X)(\phi(x))\triangleq X(x)\cdot\nabla\phi(x).
$$
 {Let $v(t)$} be a smooth vector field acting on $\RR^n$ and define its flow map by the differential equation
$$
\partial_t\psi(t,x)=v(t,\psi(t,x)),\quad \psi(0,x)=x.
$$
 It is a classical fact that for $v$ belonging to the Lipschitz class  the flow map is a diffeomorphism from $\RR^d$ to itself and thus the { push-forward} of the a vector field $X_0$ by $\psi_t$ is the vector field $(X_t)$ that can be written in the local coordinates in the form
\begin{equation}\label{fiel}
X_t(x)=\big(X_0\cdot\nabla\psi(t)\big)(\psi^{-1}(t,x)).
\end{equation}
We can easily check by using this formula that the  evolution equation governing $X_t$  is  given by the transport equation 
\begin{equation}\label{tran12}
\partial_t X+v\cdot\nabla X=X\cdot\nabla v.
\end{equation}
Besides, it is a known fact that for two smooth vector fields over $\RR^n,$ $X$ and $Y$   and for a diffeomorphism $\phi:\RR^n\to\RR^n$ we have
$$
\phi_\star[X,Y]=[\phi_\star X,\phi_\star Y].
$$
In the case where $\phi$ is given by the flow map $\psi_t$, the above identity can be easily checked using the dynamical equations. As an immediate consequence we see   that if two vector fields commute then their push-forward vector fields will commute as well. For a future use of this property it should be better to state it in the next lemma.
\begin{lem}\label{com27}
Let $(X_t)_{t\geq0}$ and $(Y_t)_{t\geq 0}$ be  two smooth vector fields solving the equation \eqref{tran12} with  the same velocity $v$. If $[X_0, Y_0]=0$, then we get
$$
[X_t, Y_t]=0,\forall t\geq 0.
$$
\end{lem} 
Next we discuss the commutation between the vector fields given by the equation \eqref{tran12} and the material derivative $D_t\triangleq \partial_t+v\cdot\nabla$ and the proof is straightforward.
\begin{prop}\label{com34}
Let $X$ be the push-forward of a smooth vector field $X_0$ defined by \eqref{tran12}. Then $X$ commutes with the transport operator $D_t\triangleq\partial_t+v\cdot\nabla$,
$$
\partial_XD_t-D_t\partial_X=0.
$$ 

\end{prop}
\subsection{Hamiltonian vector fields}
We shall discuss now some special structures of  Hamiltonian vector fields in two dimensions. To precise the terminology, we say that a smooth vector fields is Hamiltonian if it is divergence-free and in this case there exists a potential scalar function, called {\it stream function} or {\it Hamiltonian function}, $\varphi:\RR^2\to\RR$ such that
$$
X(x)=\nabla^\perp\varphi(x)\triangleq\begin{pmatrix}
-\partial_2\varphi&\\
\partial_1\varphi&
\end{pmatrix}. $$

{\bf{Notation:}} Let $X:\RR^n\to\RR^n$ be  a continuous vector field. We denote by $\mathcal{Z}_X$ the set of the zeros of $X$, that is its {\it singular} set defined by
$$
\mathcal{Z}_X=\big\{x\in \RR^n, X(x)=0\big\}.
$$
A  point $x$ is said to be {\it regular} for $X$ when $X(x)\neq 0.$
Obviously the singular set is closed and the regular one is open.

\begin{defi}
Let $\gamma:[0,1]\to \RR^2$ be a  $C^1$ Jordan curve  and $X:\RR^2\to\RR^2$ be a $C^1$  vector field. We say that $X$ is {\it co-normal}  or {\it tangential} to the curve  $\gamma$ if  $X$ is  regular on $\gamma$  and 
$$
X(x)\cdot n(x)=0,\quad\forall x\in \gamma,
$$
where $n(x)$ denotes a normal vector to the curve at the point $x$. 

\end{defi}
Sometimes we use the vocabulary  {\it streamline} or a {\it field line}  for  $X$  to denote a curve obeying to the previous definition. This terminology is justified by the next classical result.
\begin{prop}\label{prop754}
Let $\gamma:[0,1]\to \RR^2$ be a  $C^1$ Jordan curve  and $X:\RR^2\to\RR^2$ be  a $C^1$ Hamiltonian vector field. Then  $X$ is co-normal to the curve  $\gamma$ if and only if   the curve is a {\it streamline} or a level set for the potential function $\varphi$, that is there exists a constant $\lambda$ such that
$$
\varphi(x)=\lambda,\quad\forall x\in \gamma.
$$ 
In this case the vector field $X$ has at least a singular point inside the domain delimited by the \mbox{curve $\gamma.$}
\end{prop}
\begin{proof}
Denote by $t\in[0,1]\mapsto (x_1(t),x_2(t))$ a parametrization of the curve $\gamma$. Then a normal vector to the curve is given by $n=(-x_2^\prime(t), x_1^\prime(t)).$ Now  $X$ is co-normal to this curve means that for \mbox{any $t\in[0,1]$}
\begin{eqnarray*}
\nabla^\perp \varphi(x_1(t), x_2(t))\cdot (-x_2^\prime(t), x_1^\prime(t))=0.
\end{eqnarray*}
The left-side term coincides with $\frac{d}{dt}\varphi(x_1(t), x_2(t))$ and thus the co-normal assumption becomes
$$
\frac{d}{dt}\varphi(x_1(t), x_2(t))=0, \quad\forall t\in [0,1].
$$
This is equivalent to say  that $\varphi$ is constant along  the curve $\gamma.$
\end{proof}
Our next goal  is to give a precise description of the push-forward of a Hamiltonian  vector-field $X_0$ and discuss  its {\it fozen-in} property. In broad terms,  the vector fields $(X_t)$   transported by a vector field $v$ according to the equation \eqref{tran12} will remain Hamiltonian and the dynamics of the {\it stream function} will be simply described by a transport equation. This has a deep connection of the {\it freezing} of the streamlines of  vector fields $(X_t)$ into the fluid motion. This latter  property was established for the magnetic field and collectively  known  as {\it Alfv\'en's theorem}.
\begin{lem}\label{lempro}
Let $\varphi_0:\RR^2\to\RR$ be a smooth function and $X_0=\nabla^\perp \varphi_0$. Then the solution to the equation \eqref{tran12} with initial datum $X_0
$ is given by
$$
X(t,x)=\nabla^\perp \varphi(t,x)
$$
with $\varphi$ the unique solution to the problem
$$
D_t\varphi=0,\quad \varphi(0,x)=\varphi_0(x).
$$

\end{lem} 
\begin{proof}
It is straightforward computations that the vector field $x\mapsto \nabla^\perp \varphi(t,x)$ satisfies also the  equation \eqref{tran12} and thus by uniqueness of the Cauchy problem we get the desired result.

\end{proof}

\section{Vorticity-current  formulation}
Recall that the vorticity of the velocity $v$ coincides in two dimensions with the scalar function  $\omega=\partial_1 v^2-\partial_2 v^1$ and the current density of the magnetic field $b$ is given by $j=\partial_1 b^2-\partial_2 b^1.$
Applying the curl operator to  the first equation of \eqref{MHD} and using the notation $D_t=\partial_t+v\cdot\nabla$ to denote the material derivative we get
$$
D_t\omega={b\cdot\nabla j}.
$$ 
Remark that we have used the following identity: for any two-dimensional vector field $X$ we have
$$
\rot(X\cdot\nabla X)=X\cdot\nabla \rot{X}+\rot{X}\,\Div X.
$$
Performing similar computations  for the second equation of \eqref{MHD} one gets
$$
D_t j=b\cdot \nabla \omega+\big(\partial_1b\cdot\nabla v^2-\partial_2b\cdot\nabla v^1\big)-\big(\partial_1v\cdot\nabla b^2-\partial_2v\cdot\nabla b^1\big).
$$
By straightforward computations we can easily check that
\begin{eqnarray*}
\partial_1b\cdot\nabla v^2-\partial_2b\cdot\nabla v^1&=&-\big(\partial_1v\cdot\nabla b^2-\partial_2v\cdot\nabla b^1\big)+\omega\,\Div\, b+j\,\Div\, v\\
&=&-\big(\partial_1v\cdot\nabla b^2-\partial_2v\cdot\nabla b^1\big).
\end{eqnarray*}
Consequently the MHD system can be written in terms of the coupled equations on $\omega$ and $j$,
\begin{equation}
\label{E}
 \left\{ 
\begin{array}{ll} 
D_t\omega=b\cdot\nabla j\\
D_tj=b\cdot\nabla\omega+2\partial_1b\cdot\nabla v^2-2\partial_2b\cdot\nabla v^1.
\end{array} \right.    
     \end{equation}
 For reasons that will be apparent shortly in the proof of Theorem \ref{thm1} we shall need some algebraic structure especially  for the last term of \eqref{E}.   
 
 We shall start with the following identities used  in \cite{Chez,che1} and  whose proof are very  simple. Let $X=(X_1,X_2)$ be a smooth  vector field over $\RR^2,$ then
     \begin{eqnarray}\label{Iden34}
     |X(x)|^2\partial_{11}&=&X_1\partial_X\partial_1-X_2\partial_X\partial_2+X_2^2\Delta,\\
    \nonumber |X(x)|^2\partial_{22}&=&X_2\partial_X\partial_2-X_1\partial_X\partial_1+X_1^2\Delta,\\
     \nonumber |X(x)|^2\partial_{12}&=&X_1\partial_X\partial_2+X_2\partial_X\partial_1-X_1 X_2\Delta.
     \end{eqnarray}
     Applying these identities to $\Delta^{-1}\omega$ and using Biot-Savart law $\Delta v=\nabla^\perp\omega$ we get for any  $x\in \RR^2$
      \begin{eqnarray}\label{yasser}
     |X(x)|^2\mathcal{R}_{11}\omega&=&X_1\partial_X v^2+X_2\partial_X v^1+X_2^2\omega,\\
  \nonumber |X(x)|^2\mathcal{R}_{22}\omega&=&-X_2\partial_X v^1-X_1\partial_X v^2+X_1^2\omega,\\
     \nonumber |X(x)|^2\mathcal{R}_{12}\omega&=&-X_1\partial_X v^1+X_2\partial_X v^2-X_1 X_2\omega,
     \end{eqnarray}
    where we denote by $\mathcal{R}_{ij}$ the Riesz transform $\partial_{ij}\Delta^{-1}$. Therefore we obtain 
     \begin{equation}\label{vort}
     |X(x)|^2|\nabla v(x)|\lesssim \|X\|_{L^\infty}\|\partial_X v\|_{L^\infty}+\|X\|_{L^\infty}^2\|\omega\|_{L^\infty}.
     \end{equation}
     Similarly we find
     \begin{equation}\label{vorts1}
     |X(x)|^2|\nabla b(x)|\lesssim \|X\|_{L^\infty}\|\partial_X b\|_{L^\infty}+\|X\|_{L^\infty}^2\| j\|_{L^\infty}.
     \end{equation}
    The following lemma will play a crucial role in the proof of Theorem \ref{thm1}.
     \begin{lem}
     For smooth divergence-free vector fields $X, b$ and $v $ we get for $X(x)\neq 0,$
    \begin{eqnarray}\label{Iden1}
     H(v,b)&\triangleq&\partial_1b\cdot\nabla v^2-\partial_2b\cdot\nabla v^1\\
     \nonumber&=&\frac{2}{|X|^2}\Big\{\partial_X b^1\,\partial_X v^2-\partial_X b^2\,\partial_X v^1\Big\}\\
     \nonumber&+&\frac{1}{|X|^2}\Big\{j\, X\cdot\partial_X v -\omega\, X\cdot \partial_X b\Big\}.
     \end{eqnarray}
    The dot $\bf\cdot$ denotes the canonical inner product of $\RR^2.$ 
     \end{lem}
     \begin{proof}
   According to  Biot-Savart laws one has
     $$
     H(v,b)=\mathcal{R}_{12}\omega\big(\mathcal{R}_{11} j-\mathcal{R}_{22} j  \big)-\mathcal{R}_{12}j\big(\mathcal{R}_{11} \omega-\mathcal{R}_{22} \omega  \big).
     $$
     Using \eqref{yasser} we get
     \begin{eqnarray*}
     \mathcal{R}_{11} j-\mathcal{R}_{22} j &=&\frac{1}{|X(x)|^2}\Big(2X_1\partial_Xb^2+2X_2\partial_X b^1+(X_2^2-X_1^2)j \Big)
     \end{eqnarray*}
     and thus
     \begin{eqnarray*}
     \mathcal{R}_{12}\omega\big(\mathcal{R}_{11} j-\mathcal{R}_{22} j  \big)&=&\frac{1}{|X(x)|^4}\Big(-X_1\partial_X v^1+X_2\partial_X v^2-X_1 X_2\omega  \Big)\\
     &\times&\Big(2X_1\partial_Xb^2+2X_2\partial_X b^1+(X_2^2-X_1^2)j \Big).
     \end{eqnarray*}
     Similarly we get
      \begin{eqnarray*}
 \mathcal{R}_{12}j\big(\mathcal{R}_{11} \omega-\mathcal{R}_{22} \omega  \big)&=&\frac{1}{|X(x)|^4}\Big(-X_1\partial_X b^1+X_2\partial_X b^2-X_1 X_2j \Big)\\
     &\times&\Big(2X_1\partial_Xv^2+2X_2\partial_X v^1+(X_2^2-X_1^2)\omega \Big).
     \end{eqnarray*}
     Subtracting the preceding identities yields to the desired identity.
     
     \end{proof}

\subsection{Weak estimates}
In what follows we shall investigate some weak estimates for the vorticity and the current density.
     \begin{prop}\label{a priori}
    Let  $(\omega, j) $ be a smooth solution of the system \eqref{E0} then the following results hold true.
     \begin{enumerate}
     \item For  $\omega_0,j_0\in L^p$ with $1<p<\infty$ we get 
     $$
     \|(\omega,j)(t)\|_{L^p}\le C\|(\omega_0, j_0)\|_{L^p}e^{C\int_0^t\|\nabla v(\tau)\|_{L^\infty}d\tau}.
     $$

     \item For $\omega_0,j_0\in L^\infty$ we get
     $$
     \|(\omega,j)(t)\|_{L^\infty}\le C\|(\omega_0, j_0)\|_{L^\infty}+\int_0^t\|\nabla v(\tau)\|_{L^\infty}\|\nabla b(\tau)\|_{L^\infty}d\tau.
     $$
       \item Let $\omega_0,j_0\in L^1\cap L^2$, we get
     $$
      \|(\omega,j)(t)\|_{L^1}\le C \|(\omega_0, j_0)\|_{L^1}+C\|(\omega_0, j_0)\|_{L^2}^2 t\,e^{C\int_0^t\|\nabla v(\tau)\|_{L^\infty}d\tau}.
     $$
     \end{enumerate}
     \end{prop}
     \begin{proof}
    ${\bf{(1)}}$  Applying Proposition \ref{transport} to the equation \eqref{E} we get
     $$
     \|\omega(t)\|_{L^p}+\|j(t)\|_{L^p}\lesssim   \|\omega_0\|_{L^p}+\|j_0\|_{L^p}+\int_0^t\|\nabla b(\tau)\|_{L^p}\|\nabla v(\tau)\|_{L^\infty}d\tau
     $$
     Using the continuity of Riesz transform on $L^p$ with $p\in]1,\infty[$ one gets
     $$
     \|\nabla b(t)\|_{L^p}\lesssim \| j(t)\|_{L^p}
     $$
    which  yields in turn
     $$
          \|\omega(t)\|_{L^p}+\|j(t)\|_{L^p}\lesssim   \|\omega_0\|_{L^p}+\|j_0\|_{L^p}+\int_0^t\| j(\tau)\|_{L^p}\|\nabla v(\tau)\|_{L^\infty}d\tau.
$$
It suffices now to apply Gronwall inequality in order to get the suitable estimate.
\vspace{0,2cm}

${\bf{(2)}}$ Using once again Proposition  \ref{transport} implies
$$
         \|\omega(t)\|_{L^\infty}+\|j(t)\|_{L^\infty}\lesssim   \|\omega_0\|_{L^\infty}+\|j_0\|_{L^\infty}+\int_0^t\|\nabla b(\tau)\|_{L^\infty}\|\nabla v(\tau)\|_{L^\infty}d\tau.
$$
which is the desired result.
\vspace{0,2cm}

${\bf{(3)}}$ Arguing as before and using H\"{o}lder inequality we obtain
$$
 \|\omega(t)\|_{L^1}+\|j(t)\|_{L^1}\lesssim   \|\omega_0\|_{L^1}+\|j_0\|_{L^1}+\int_{0}^t\|\omega(\tau)\|_{L^2}\|j(\tau)\|_{L^2}d\tau
$$
At this stage we combine this estimate with the one of the first part $(1)$.
     \end{proof}
     
\section{Stationary patches}\label{stat}
As we can readily  observe from the vorticity-current formulation \eqref{E0} the structure of the initial patches $\omega_0=\chi_\Omega, j_0=\chi_D$ cannot be in general  conserved in time in contrast with the incompressible Euler equations. This is due peculiarly to the last two terms in the second equation  involving Riesz transforms. In what follows we shall look for stationary solutions for \eqref{MHD} in the framework of vortex patches. In other words, we shall characterize the simply connected  bounded domains $\Omega$ and $D$ such that   $\omega(t)=\chi_{\Omega}$ and $j(t)=\chi_D$ is  a solution for the system \eqref{E0}.  
First observe that when the domains are concentric balls then according to the symmetry invariance of the equations  we obtain a stationary solution. We will see that in the case of the disjoint patches  these are the only examples of  stationary  solutions. The proof that we shall present  of this intuitive result  is not trivial  but it will make appeal  to a deep result of potential theory which characterize the circle with Newtonian potential. Our result which was introduced in Theorem \ref{The44} will be now  restated only for the inviscid MHD system.
\begin{Theo}\label{The47}
Let $D$ and $\Omega$ be two simply connected domains  and $\omega_0={\chi}_{\Omega}$,   $j_0={\chi}_{D}.$ Then the following holds true:
\begin{enumerate}
\item If $D=\Omega$ then $(\omega_0, j_0)$ is a stationary solution for the MHD system \eqref{MHD}.
\item If the boundaries $\partial D$ and $ \partial\Omega$ are disjoint rectifiable Jordan curves   then $(\omega_0, j_0)$ is a stationary solution for the MHD system \eqref{MHD} if and only if  $\Omega$ and $D$ are concentric balls.

\end{enumerate}
\end{Theo}
Some remarks are in order.

\begin{rema}
We can deduce from the proof that for  the two-dimensional incompressible Euler equations the  stationary patches with rectifiable Jordan boundaries are given by   $\chi_{\Omega}$ with $\Omega$ a ball.
\end{rema}
\begin{rema}
In the case of the $2d$ incompressible Euler equations we know according to Yudovich result that the patches give rise to unique global solutions. Whether or not  the same claim remains true for   the MHD equations  even locally in time is not at all clear.  We will see later in the next section that this can be   proved  for example for the patches with sufficiently smooth boundaries.
 \end{rema}
The proof of Theorem \ref{The47} relies on Franekel's result and will be divided into two steps depending on the smoothness of the boundaries. The case of  $C^1$ boundaries is more easier than the rectifiable ones and    we shall need  for this latter case more sophisticated analysis. Especially we  will use  the conformal mappings to parametrize the boundaries combined with some interesting properties on their boundary behavior.  For the clarity of the proofs it would be better  to   recall some basic results on conformal mappings and rectifiable Jordan curves that will be substantially   used later. This will be the subject of the next section.
\subsection{Conformal mappings}\label{conf-m}
We shall in the first part fix some notation and concepts. Afterwards we discuss the conformal mapping theorem and some basic results on the boundary  behavior of the conformal maps.  

\quad A planar curve $C$ is called  a {\it Jordan curve} if it is simple  and closed meaning that it can be parametrized by an injective continuous function $\gamma:\mathbb{T}\to \RR^2.$
This curve is said to be {\it rectifiable} if it is of {\it bounded variation} and its length $L$ is the total variation of $\gamma$. This means that
$$
L\triangleq \sup_{(\xi_i)_{i=1}^n\in \mathcal{P}}\sum_{k=1}^n|\gamma(\xi_{k+1})-\gamma(\xi_k)|<\infty
$$
where the supremum is taken over all the partition $\mathcal{P}$ of the unit circle $\mathbb{T}$.
\vspace{0,2cm}

The  following result due to Riemann is one of the most important results in complex analysis. To restate this result we shall recall the definition of {\it simply connected} domains. Let $\widehat{\mathbb{C}}\triangleq \mathbb{C}\cup\{\infty\}$ denote the Riemann sphere. We  say that a domain $\Omega\subset \widehat{\mathbb{C}}$ is {\it simply connected} if the set $ \widehat{\mathbb{C}}\backslash \Omega$ is connected.

{\bf Riemann Mapping Theorem.} Let $\mathbb{D}$ denote the unit open ball, $\Omega\subset \CC$ be a simply connected domain different from  $\CC$ and $z_0\in \Omega$. Then there is a unique bi-holomorphic map (conformal map) $\Phi:\mathbb{D}\to \Omega$ such that
$$
\Phi(0)=z_0\quad\hbox{and}\quad \hbox{arg } \Phi^\prime(0)>0.
$$
The area of the domain $\Omega$ is given by
$$
|\Omega|=\int_{\mathbb{D}}|\Phi^\prime(z)|dA(z),
$$
where $dA$ denotes the Lebesgue measure of the plane. In this theorem the regularity of the boundary has no effect regarding the existence of the conformal mapping  but as it was shown in various papers   it  will contribute in the boundary behavior of the conformal mapping, see for instance \cite{P,WS}. One of the main result in this subject  dealing with the continuous extension to the boundary goes back to Carath\'eodory. 
\vspace{0,2cm}

{\bf Carath\'eodory Theorem}. The conformal map $\Phi: \mathbb{D}\to \Omega$ has one-to-one continuous  extension to  the closure $\overline{\mathbb{D}}$ if and only if the boundary $\partial \Omega$ is a Jordan curve.
\vspace{0,2cm}

In the next theorem we shall discuss the characterization of rectifiable Jordan curves in terms of the regularity of the associated conformal map. This will require the use of Hardy space of \hbox{type $H^1$} which is defined as follows.   Let $f:\mathbb{D}\to \mathbb{C}$ be an analytic function, we define the integral means
$$
M(r,f)\triangleq \int_0^{2p}|f(r e^{i\theta})|d\theta,\, 0<r<1.
$$
The function $f$  is said to be of class $H^1$ if 
$$
\sup_{0<r<1}M(r,f)<\infty.
$$
A classical result known by the name {\it Hardy's convexity theorem} asserts that $r\mapsto M(r,f)$ is  a nondecreasing function and  $r\mapsto \log M(r,f)$ is a convex function of $\log r$. For the proof of this result see for example Theorem $1.5$ of  \cite{Duren}, a reference which provides  additional relevant information on the topic.

 Next we  shall give an analytic characterization of   rectifiable curves through the regularity of the  conformal mapping.
\begin{Theo}\label{Hp}

The following assertions hold true.

{\bf{$\hbox{1})$}} 
Let $\Phi: \mathbb{D}\to \Omega$ be  the conformal mapping. Then $\partial\Omega$ is rectifiable if and only if $\Phi^\prime\in H^1$ and
$$
L=\lim_{r\to 1}\int_{0}^{2p}|\Phi^\prime(r e^{i\theta})| d\theta.
$$
{\bf{$\hbox{2})$}}  Let $f\in H^1$ then $f$ has an angular limit $f(e^{i\theta})$ almost everywhere on the boundary $\mathbb{T}$ and
$$
\lim_{r\to 1}\int_{0}^{2p}|f(r e^{ i\theta})|d\theta=\int_{0}^{2p}|f( e^{ i\theta})|d\theta,\quad\hbox{and}\quad \lim_{r\to 1}\int_{0}^{2p}|f(r e^{ i\theta})-f(e^{i\theta})|d\theta=0.
$$

\end{Theo}
The first result is discussed in Theorem $3.12 $  of \cite{Duren}. As to the second one, we refer  the reader to Theorem 2.6 of the same reference.
An immediate  consequence of  Theorem \ref{Hp} reads as follows.
\begin{Coro}\label{corSZ}
Let $\Phi$ be a conformal mapping of the unit ball $\mathbb{D}$ onto the interior of a rectifiable Jordan curve $\partial\Omega$. Then $\Phi^\prime$ has an angular limit almost everywhere on  the boundary $\mathbb{T}$ and 
$$
\lim_{r\to 1}\int_{0}^{2p}|\Phi^\prime(r e^{ i\theta})|d\theta=\int_{0}^{2p}|\Phi^\prime( e^{ i\theta})|d\theta,\quad\hbox{and}\quad \lim_{r\to 1}\int_{0}^{2p}|\Phi^\prime(r e^{ i\theta})-\Phi^\prime(e^{i\theta})|d\theta=0.
$$
\end{Coro}
\subsection{Potential characterization of the balls}
There are many results emerging from potential theory with the basic goal to characterize the balls of the Euclidian space $\RR^n.$ One of them uses the Newtonian potential defined for a domain $\Omega\subset \RR^2$ by
\begin{equation}\label{logp}
\varphi(x)=\frac{1}{2p}\int_{\Omega}\log\frac{1}{|x-y|}dy.
\end{equation}
In the vocabulary of fluid dynamics this is the stream function of the vorticity $\chi_{\Omega}$. When the domain coincides with a ball then $\varphi$ is constant on the boundary. The converse is proved by  Fraenkel  \cite{Fran}, see Theorem 1.1 page $18,$ that we recall here.

\begin{Theo}\label{fran}
Let $\Omega$ be  a bounded domain set of  $\RR^2$  and $\varphi$ its Newtonian potential. If $\varphi$ is constant on the boundary of $\Omega$, then $\Omega$ must be a ball.
\end{Theo}
The result of Fraenkel is not specific to the two dimensions  but can be extended for higher dimensions. Moreover it is worth pointing out that this theorem does not require any  assumption on the regularity of the boundary. Recently a partial extension of this result was  accomplished by Reichel  in \cite{Rei}.
\subsection{Proof of Theorem \ref{The47}}  We intend to give the proof concerning the stationary patches.

\begin{proof}
$(\bf{1})$ This can be deduced from the following fact which is related to the special structure of the inviscid MHD equations \eqref{MHD}: if $b_0=\pm v_0$ then we can easily check that  this  corresponds to a stationary solution without pressure. This allows to get the desired result.

$(\bf{2})$ This proof   is more tricky and  founded  on   Theorem \ref{fran}.  For the sake of clear presentation we shall distinguish  smooth boundaries from the rough ones. We mean by {\it  smooth } a $C^1$ Jordan curve and by {\it  rough}   a rectifiable Jordan curve. As we shall see the basic difference between these cases appears when we deal with  the flux across the boundary.  For smooth boundaries this can be done by  using Gauss-Green formula.  However for the rough boundaries more sophisticated analysis will be required. To answer to this problem there are at least two approaches that one could consider. The first one is to use a general version  of Gauss-Green formula coming from the geometric measure theory. The disadvantage  of this  formula is that it not so explicit to allow  exploitable computations.  The second one that will be developed  here is to use the conformal mappings. Therefore   the problem reduces to measuring the flux across the unit sphere for a modified vector field and by this way we transform the problem into the regularity of the conformal mapping close to the boundary. This has been discussed previously in Section \ref{conf-m}.  

\vspace{0,2cm}

$\bullet$ {\bf Smooth curves.} We assume that the curves  $\partial\Omega$ and $\partial D$ are of class  $C^1$, then we may use the  classical result concerning the derivative in the distribution sense of the characteristic function  $\chi_D,$

\begin{equation}\label{deriv}
\nabla{\chi_{D}}=-\vec{n} \,d\sigma_{\partial D},
\end{equation}
where $d\sigma_{\partial D}$ is the arc-length measure on $\partial D$ and
$\vec{n}$ the outward unit normal. 
Accordingly the first equation of \eqref{E} can be written for the stationary patches in  the form
$$
(v_0\cdot\vec n)d\sigma_{\partial \Omega}=(b_0\cdot\vec n)d\sigma_{\partial D}.
$$
Since the boundaries are disjoint then the involved measures are disjointly supported and thus,
\begin{equation}\label{ah12}
v_0\cdot\vec n=0\quad\hbox{on }\quad \partial\Omega\qquad\hbox{and}\qquad b_0\cdot\vec n=0\quad\hbox{on }\quad \partial D.
\end{equation}
Denote by $\varphi_0$ and $\psi_0$ the stream functions of $v_0$ and $b_0$, respectively. They satisfy the elliptic equations
$$
\Delta\varphi_0=\chi_{\Omega},\quad \Delta\psi_0=\chi_{D}.
$$
Now since $v_0=\nabla^\perp\varphi_0$ and $b_0=\nabla^\perp \psi_0$ we deduce in view of Proposition \ref{prop754} that the stream functions $\varphi_0$ and $\psi_0$ are constant on the boundaries $\partial\Omega$ and $  \partial D,$ respectively. At this stage we can use Fraenkel's theorem and conclude  that $\Omega$ and $D$ are balls. It remains to show that these  balls are concentric. For this goal we will use the second equation of \eqref{E}. Thus performing similar calculations we get in the weak sense,
$$
(v_0\cdot\vec n)d\sigma_{\partial D}=(b_0\cdot\vec n)d\sigma_{\partial \Omega}-\big\{2\partial_1b_0\cdot\nabla v_0^2-2\partial_2b_0\cdot\nabla v_0^1\big\}.
$$
As the functions $\nabla v_0, \nabla b_0$ belong to $ L^p$ for any finite $p\in (1,\infty)$, the last term appearing  between the braces is a function and consequently the preceding equation is equivalent to the conditions 
\begin{equation}\label{ae}
v_0\cdot\vec n=0\quad\hbox{on }\quad \partial D;\quad b_0\cdot\vec n=0\quad\hbox{on }\quad \partial\Omega
\end{equation}
and 
\begin{equation}\label{aew}
\partial_1b_0\cdot\nabla v_0^2-\partial_2b_0\cdot\nabla v_0^1=0,\quad \hbox{a.e.}
\end{equation}
 Consequently we deduce by  Proposition \ref{prop754} that the  curve $\partial D$ is a streamline for $\varphi_0$ and $\partial\Omega$ is a streamline for $\psi_0.$ Without loss of generality one can assume that the ball $\Omega$ is centered at the origin and with \mbox{radius $r$.}
It is known that  in this case the stream function $\varphi_0$ has  the form
\begin{equation*}
 \varphi_0(x)=\left\{ 
\begin{array}{ll} 
\frac14|x|^2-\frac{r^2}{2}\big(\log\frac1r+ \frac12\big),\quad  |x|\le r,\\
\vspace{0,1cm}

\frac{r^2}{2}\ln{|x|},\quad |x|\geq r.
\end{array} \right.    
     \end{equation*}
From which we deduce that the streamlines of $\varphi_0$ are concentric circles and this implies in turn  that $\partial D$ has the same center as $\partial \Omega.$ We can also deduce that 
$$
v_0(x)=f(|x|) x^\perp\quad \hbox{and}\quad b_0(x)=g(|x|) x^\perp
$$
for  two known Lipschitz  functions $f$ and $g$. To check the equation \eqref{aew} we shall write it in the weak sense and use the foregoing structure for $v_0$ and $b_0,$ 
\begin{eqnarray*}
\partial_1b_0\cdot\nabla v_0^2-\partial_2b_0\cdot\nabla v_0^1&=&\textnormal{div}\big(v_0^2 \partial_1b_0-v_0^1\partial_2b_0 \big)\\
&=&\textnormal{div}\big(r f(r)\partial_r(g(r) x^\perp)\big)\\
&=&\textnormal{div}\big( f(r)(r\partial_rg+g) x^\perp\big)\\
&=&0.
\end{eqnarray*}

This concludes the proof in the case of disjoint $C^1$ boundaries. 

\vspace{0,2cm}

$\bullet$ {\bf Rough curves.}
We start with writing in the distribution sense the equations of the stationary patches according to the equations \eqref{E}. The first one reads as follows
$$
\textnormal{div}(v\,\chi_{\Omega})=\textnormal{div}(b\,\chi_{D}).
$$
We can readily  check in view of the incompressibility condition that the supports of these distributions satisfy
$$
\hbox{supp}(\textnormal{div}(v\,\chi_{\Omega}))\subset \partial\Omega\quad\hbox{and}\quad \hbox{supp}(\textnormal{div}(v\,\chi_{D}))\subset \partial D.
$$
Since $\partial D\cap \partial \Omega=\varnothing$ we get
$$
\textnormal{div}(v\,\chi_{\Omega})=\textnormal{div}(b\,\chi_{D})=0.
$$
This means that for any $\psi\in \mathcal{D}(\RR^2),$ we have
$$
\int_{\Omega}\textnormal{div}(v\,\psi)dx=\int_{D}\textnormal{div}(b\psi)dx=0.
$$
The next goal is to deduce from these equations that the Newtonian potential $\varphi_0$ and $\psi_0$ introduced in the previous case are constant on the corresponding boundaries. With this in hand we can conclude by using Fraenkel's result and deduce that the boundaries are necessary circles. Afterwards we shall check by using similar arguments as previously  that their centers must agree. As we have mentioned before the major difficulty concerns the   use of  Gauss-Green formula for rough boundaries. Our approach is based on the use   the conformal mapping combined with an approximation procedure. 
Let $\mathbb{D}$ denote the unit open ball and  $\Phi: \mathbb{D}\to \Omega$ be  a conformal mapping. Since $\partial \Omega$ is a Jordan curve   then by  Carath\'eodory theorem $\Phi$ has a continuous extension to $\overline{\mathbb{D}}$ and  maps the unit circle $\mathbb{T}$ one-to-one onto $\partial \Omega.$   Moreover, according to Theorem  \ref{Hp}  as the boundary is rectifiable Jordan curve  the derivative $\Phi^\prime$ exists for  almost all $\xi\in \mathbb{T}$ and 
$$
\Phi^\prime\in L^1(\mathbb{T}).
$$
First recall  the  Cauchy-Riemann equations for the holomorphic function  $\Phi=\Phi_1+i\Phi_2$ inside $\mathbb{D}$ 
$$
\partial_1\Phi_1=\partial_2\Phi_2,\quad \partial_2\Phi_1=-\partial_1\Phi_2.
$$
Set $F\triangleq v\,\psi$, then we have the equation
\begin{equation}\label{eq7d}
\int_{\Omega}\textnormal{div}\,F\,\, dx=0.
\end{equation}
Observe that from Biot-Savart law the velocity enjoys the following regularities
$$
v\in L^\infty,\,\nabla v\in L^p, \,\forall p\in ]1,\infty[.
$$
 For each $0<r<1$, let us denote by  $\mathbb{D}_r$   the open  ball  of radius $r$ and centered at the origin and set 
 $$\Omega_r\triangleq\Phi(\mathbb{D}_r).
 $$ 
The boundary $\partial\Omega_r$ is an analytic curve and  $\Phi$ maps conformally $\mathbb{D}_r$ onto $\Omega_r.$    By the change of variable $x=\Phi(y)$ we obtain
\begin{equation}\label{hmi650}
\int_{\Omega_r}\textnormal{div}\,F\,\, dx=\int_{\mathbb{D}_r}(\textnormal{div}\,F)(\Phi(y))\,\, |J_\Phi(y)|dy.
\end{equation}
Using  Cauchy-Riemann equations we obtain the following formula for  the Jacobian $|J_\Phi(y)|$,
$$
|J_\Phi(y)|=(\partial_1\Phi_1(y))^2+(\partial_2\Phi_1(y))^2,\quad \forall \, y\in \mathbb{D}.
$$
Now we claim that,
\begin{eqnarray}\label{eq907}
(\textnormal{div}\,F)(\Phi(y))\,\, |J_\Phi(y)|&=&\partial_1\Phi_1(\textnormal{div}\,(F(\Phi(y))+\partial_1\Phi_2\,\hbox{curl}(F(\Phi(y))\\
\nonumber&=&\partial_2\Phi_2\partial_1(F_1(\Phi(y))-\partial_1\Phi_2\partial_2(F_1(\Phi(y))\\
\nonumber&+& \partial_1\Phi_1\partial_2(F_2(\Phi(y))-\partial_2\Phi_1\partial_1(F_2(\Phi(y)).
\end{eqnarray}
Indeed, easy computations yield

\begin{equation*}
\left\{ 
\begin{array}{ll} 
\partial_1(F_j(\Phi))=(\partial_1F_j)(\Phi)\partial_1\Phi_1+(\partial_2F_j)(\Phi)\partial_1\Phi_2\\
\vspace{0,1cm}

\partial_2(F_j(\Phi))=(\partial_1F_j)(\Phi)\partial_2\Phi_1+(\partial_2F_j)(\Phi)\partial_2\Phi_2.
\end{array} \right.    
     \end{equation*}
Using Cauchy-Riemann equations one finds
\begin{equation*}
\left\{ 
\begin{array}{ll} 
(\partial_1F_j)(\Phi)|J_\Phi|=\partial_1\Phi_1\partial_1(F_j(\Phi))-\partial_1\Phi_2\partial_2(F_j(\Phi))\\
\vspace{0,1cm}

(\partial_2F_j)(\Phi)|J_\Phi|=\partial_2\Phi_2\partial_2(F_j(\Phi))-\partial_2\Phi_1\partial_1(F_j(\Phi)).\end{array} \right.    
     \end{equation*}
     This gives the identity \eqref{eq907}. Combining \eqref{hmi650} and \eqref{eq907} with Gauss-Green formula we get
     \begin{eqnarray*}
     \int_{\Omega_r}\textnormal{div}\,F\,\, dx&=&
    r \int_{0}^{2p}\Big((\partial_2\Phi_2)(r e^{i\theta})\cos\theta-(\partial_1\Phi_2)(r e^{i\theta})\sin\theta\Big)F_1(\Phi(re^{i\theta}))d\theta\\
       &+&r\int_{0}^{2p}\Big((\partial_1\Phi_1)(re^{i\theta})\sin\theta-(\partial_2\Phi_1)(re^{i\theta})\cos\theta\Big)F_2(\Phi(re^{i\theta}))d\theta.
     \end{eqnarray*}
    
     Recall that 
     $$
     F=v\psi=\psi\,\nabla^\perp \varphi_0 
     $$
     and thus with the notation $\zeta=e^{i\theta}$ we get
      \begin{eqnarray*}
     \int_{\Omega_r}\textnormal{div}\,F\,\, dx
     &=&-r\int_{0}^{2p}\Big((\partial_2\Phi_2)(r \zeta)\cos\theta-(\partial_1\Phi_2)(r \zeta)\sin\theta\Big)(\partial_2\varphi_0)(\Phi(r\zeta))\psi(\Phi(r\zeta))d\theta\\
       &+&r\int_{0}^{2p}\Big((\partial_1\Phi_1)(r \zeta)\sin\theta-(\partial_2\Phi_1)(r \zeta)\cos\theta\Big)(\partial_1\varphi_0)(\Phi(r\zeta))\psi(\Phi(r\zeta))d\theta\\
       &=&-r\int_{0}^{2p}\Big((\partial_2\Phi_1)(r \zeta)(\partial_1\varphi_0)(\Phi( r\zeta))+(\partial_2\Phi_2)(r \zeta)(\partial_2\varphi_0)(\Phi(r\zeta))\Big)\cos\theta\,\psi(\Phi(r\zeta))d\theta\\
       &+&r\int_{0}^{2p}\Big((\partial_1\Phi_1)(r \zeta)(\partial_1\varphi_0)(\Phi(r\zeta))+(\partial_1\Phi_2)(r \zeta)(\partial_2\varphi_0)(\Phi(r\zeta))\Big)\sin\theta\,\psi(\Phi(r\zeta))d\theta \\
       &=&-\int_{0}^{2p}\frac{d}{d\theta}\big\{\varphi_0(\Phi(re^{i\theta}))\big\}\psi(\Phi(re^{i\theta}))d\theta.
     \end{eqnarray*}
 We have used the notation,
 $$
 \varphi_0(\Phi(re^{i\theta}))=\varphi_0\big(Q_1(r\cos\theta, r\sin\theta),Q_2(r\cos\theta, r\sin\theta)  \big).
 $$To pass to the limit in the left-hand side when $r$ approaches $1$ we use  that 
 $$\textnormal{div F}=v\cdot\nabla \psi\in L^\infty
 $$ combined with the fact that the area of $\Omega_r$ converges to the area of $\Omega.$ This latter claim follows from the formula
 $$
 |\Omega\backslash\Omega_r|=\int_{\mathbb{D}\backslash\mathbb{D}_r}|\Phi^\prime(z)|dA(z).
 $$
Consequently
\begin{eqnarray}\label{eq7d}
 \nonumber \lim_{r\to 1}\int_{\Omega_r}\textnormal{div}\,F\,\, dx&=&\int_{\Omega}\textnormal{div} F\, dx\\
 &=&0.
 \end{eqnarray}

Concerning the passage to the limit in the right-hand side we use  $\nabla\varphi_0\in \mathcal{C}_b(\RR^2)$ combined with the following result discussed before in the preliminaries,
 $$
 \lim_{r\to 1}\int_{0}^{2p}|\nabla\Phi(re^{i\theta})-\nabla\Phi(e^{i\theta})|d\theta=0.
 $$  
    Therefore we obtain from the preceding identity combined with  \eqref{eq7d}
     \begin{equation}\label{him1}
     \int_{0}^{2p}\frac{d}{d\theta}\big\{\varphi_0(\Phi(e^{i\theta}))\big\}\psi(\Phi(e^{i\theta}))d\theta=0,\quad\forall\,\psi\in \mathcal{D}(\RR^2).
     \end{equation}
     Let $h:\mathbb{T}\to\CC$ be any  continuous function on the circle. From Carath\'eodory Theorem  we can extend $\Phi$ to the closure  $\overline{\mathbb{D}}$ and $\Phi:\overline{\mathbb{D}}\to\overline\Omega$ is a homeomorphism. Consequently   the inverse $\Phi^{-1}:\partial\Omega\to {\mathbb{T}}$ is continuous and we can define the function $\psi: \partial\Omega\to \CC$  by
     $$
     \psi(z)=h(\Phi^{-1}(z),\quad z\in \partial\Omega.
     $$
     The function $\psi$ is continuous over $\partial \Omega$ and has an extension belonging to $\mathcal{C}^\infty_c(\RR^2).$ Therefore we deduce from \eqref{him1} that for any $h\in \mathcal{C}(\mathbb{T}; \CC)$,
     $$
          \int_0^{2p}\frac{d}{d\theta}\big\{\varphi_0(\Phi(e^{i\theta}))\big\}h(e^{i\theta})d\theta=0.
     $$
     This allows to conclude that 
     $$
     \frac{d}{d\theta}\big\{\varphi_0(\Phi(e^{i\theta}))\big\}=0,\quad\hbox{a.e.}
     $$
     As $\theta\mapsto \frac{d}{d\theta}\big\{\varphi_0(\Phi(e^{i\theta}))\big\}$ is absolutely continuous then we can use  Taylor formula which  implies the existence of a constant $\lambda$ such  that
     $$
     \varphi_0(\Phi(e^{i\theta}))=\lambda, \quad\forall \theta\in [0,2p].
     $$
     This means that 
     $$
     \varphi_0(z)=\lambda,\quad\forall z\in \partial \Omega.
     $$
     At this stage we can use Fraenkel's result to conclude that the domain $\Omega$ should be a ball. The same proof shows that $D$ is also a ball and to  check that the balls have the same center we follow the same computations as for the smooth boundaries.
    
\end{proof}

     \section{Generalized vortex patches}
     
     In this section, we shall extend the conclusion of  Theorem \ref{thm1} to more general initial data belonging to the Yudovich class. 
     
     \subsection{General statement}
     
 Before stating our result we shall recall  some definitions   that  were briefly introduced in Section \ref{Sec2}. For a continuous vector field $X:\RR^2\to\RR^2$ and  $\delta\geq0$ we denote by 
$$\mathcal{Z}_X^\delta\triangleq\big{\{}x\in \RR^2,\quad |X(x)|\le \delta\big{\}}.
$$
Let $\varepsilon\in]0,1[$ and $f,g:\RR^2\to\RR$ be  two functions. We define the $C^\EE$ singular support of the couple $(f,g)$ by $$
\Sigma_{\textnormal{sing}}^\EE(f,g)\triangleq \Sigma_{\textnormal{sing}}^\EE(f)\cup \Sigma_{\textnormal{sing}}^\EE(g),
$$
where the  singular support of a single function was given in Definition \ref{singsuppp}.
 We may also  recall the push-forward  of  a  vector field $X_0$ by the flow map associated to another time dependent vector field $v$,  as the  solution of  the transport equation 
 $$
 D_t X=X\cdot\nabla v,\quad X(0)=X_0.
 $$
 Finally recall that the anisotropic space $W^{p}_X$ was introduced in Definition \ref{Defin79}.
Our main result reads as follows.
\begin{Theo}\label{thm2}
Let $p\in]2,\infty[$, $\varphi_0:\RR^2\to \RR$ be an element of    $ W^{2,\infty}$   and  $G:\RR\to \RR$ be a function such that  $G^\prime\in W^{2,\infty}$   and $\inf_{\RR}|G^\prime|>0$. Take $X_0=\nabla^\perp\varphi_0$ and  $b_0=\nabla^\perp \{G(\varphi_0)\}, $ and assume that the initial data $v_0$ and $b_0$ satisfy
\begin{enumerate}
\item  The vorticity $\omega_0$  and the  current density $j_0$ belong to $W^p_{X_0}.$
\item There exists a smooth compactly supported function $\rho$ such that $(1-\rho)\omega_0,(1-\rho) j_0\in C^{1-\frac2p}.$ 
\item The singular set  $\Sigma_{\textnormal{sing}}^{1-\frac2p}(\omega_0,j_0)$ is compact and there exists $\delta>0$ such that  
\begin{equation}\label{separa}
\textnormal{dist} \Big(\varphi_0(\mathcal{Z}_{X_0}^\delta),\varphi_0\big(\Sigma_{\textnormal{sing}}^{1-\frac2p}(\omega_0,j_0)\big)\Big)>0.
\end{equation}
\end{enumerate} 
Then there exists $T>0$ and a unique solution $(v,b)$ for the MHD equations with 
$$\forall t\in[0,T],\quad\omega(t),j(t)\in W^p _{X(t)}\quad\hbox{and}\quad X, v,b\in L^\infty([0,T]; \textnormal{Lip}).
$$
Moreover, let $\psi$ be the flow of $v$, then 
$$
\partial_{X_0}\psi(t)\in L^\infty([0,T]; \textnormal{Lip}).
$$
\end{Theo}
\begin{rema}\label{rem112}
 Let $\omega_0=\chi_{\Omega}$,   $j_0=\chi_{D}$ and $b_0$ the magnetic field associated to $j_0.$ Then following the proof of Theorem \ref{The47} we obtain that $\omega_0, j_0\in W^p_{b_0}$ if and only if $\Omega$ and $D$ are concentric discs.
\end{rema}
\begin{rema}
 Theorem \ref{thm2}  allows to work with more general vortices than the vortex patches: we can for example take an initial vorticity  of the form
 $$
 \omega_0(x)=f(x)\chi_{\Omega}
 $$
 with $f$ a smooth  compactly supported function and the magnetic field $b_0$ should  of course satisfy the assumption \eqref{separa}.
\end{rema}
As we shall see now this theorem allows to get the result of Theorem \ref{thm1}.
\subsection{Proof of Theorem \ref{thm1}}
We shall apply the preceding theorem with $\omega_0=\chi_\Omega$ and $G(x)=x$, meaning in this case  that $X_0=b_0.$ According to the assumption $j_0\in L^1\cap W^{1,p}$ and the embedding $W^{1,p}\hookrightarrow L^\infty$ we easily deduce   that $j_0\in W^p_{X_0}.$ The assumption  $\omega_0\in  W^p_{X_0}$ is equivalent to the vanishing of the normal component of the magnetic field: $b_0\cdot n=0$ on the boundary $\partial\Omega.$  Therefore it remains to check that the condition  \eqref{Equiv1} of Theorem \ref{thm1} implies the assumption \eqref{separa}. \\ From the embedding $W^{1,p}\hookrightarrow C^{1-\frac2p}$,  the $C^{1-\frac2p}$-singular support  of $j_0$ is empty and thus  the joint singular support  $\Sigma_{\textnormal{sing}}^\EE(\omega_0,j_0)$ coincides with $\partial\Omega.$    

Let $x\in \mathcal{Z}_{X_0}^\delta$ and $y\in\partial\Omega$. Since $|b_0(x)|\le\delta$ then from the condition $\eqref{Equiv1}$ of Theorem \ref{thm1} we obtain
$$
|\varphi_0(x)-\lambda|\geq\eta,
$$
with $\lambda$ the constant value of $\varphi$ on the   boundary $\partial\Omega$. It follows that 
$$
\forall x\in \mathcal{Z}_{X_0}^\delta, \forall y\in \partial\Omega,\quad |\varphi_0(x)-\varphi_0(y)|\geq\eta. 
$$
This gives the assumption \eqref{separa} and thus we can apply Theorem \ref{thm2} leading to the  first part of Theorem \ref{thm1}. Concerning the persistence regularity for the boundary $\psi(t,\partial\Omega),$ recall that a parametrization of $\Omega$ is given by the equation
$$
\partial_s\gamma_0(s)=X_0(\gamma_0(s)),\quad\gamma_0(0)=x_0\in \partial\Omega.
$$
Therefore we may  parametrize $\psi(t,\partial\Omega)$ by $\gamma_t: s\mapsto\psi(t,\gamma_0(s))$ and thus
$$
\partial_s\gamma_t(s)=(\partial_{X_0}\psi)(t,\gamma_0(s)).
$$
This implies  $\partial_s\gamma_t\in W^{1,\infty}$ and consequently $\gamma_t\in W^{2,\infty}.$
     
     \subsection{Persistence of the co-normal regularity}
     
     Next, we shall study the persistence regularity  of the solutions in the anisotropic  spaces $W_X^p$ and $C_X^\varepsilon.$ This step requires that the first  commutator between the vector field $X_0$ and  the magnetic field $b_0$ vanishes an we believe that this algebraic   condition  is not just a technical artifact  but a deep geometric obstruction for the well-posedness problem. We intend to  prove the following results.
     \begin{prop}
    \label{prop1}
    Let $\varphi_0:\RR^2\to \RR$ be an element of  $ W^{2,\infty}$  and   $G:\RR\to \RR$ be a function such that  $G^\prime\in W^{2,\infty}$   with $\inf_{\RR}|G^\prime|>0$. Take $X_0=\nabla^\perp\varphi_0$ and  $b_0=\nabla^\perp \{G(\varphi_0)\}, $ and assume that the initial data $\omega_0, j_0\in W^p_{X_0}$, with $p\in]2,\infty[$.
           Let $(v,b)$ be a smooth solution of the system \eqref{MHD} defined in some interval $[0,T],$ with $T\leq1$ and  $(X_t)$ be the push-forward of the vector field $X_0$. Then for \mbox{any $t\in [0,T],$  }
     $$
\|(\partial_X\omega,\partial_X j)(t)\|_{L^p}\le C_0\big(1+t W^2(t)\big)e^{CtW(t)}
$$
and 
$$
\| (\partial_X v,\partial_X b)(t)\|_{C^{1-\frac2p}}+\|X_t\|_{C^{1-\frac2p}}\le C_0\big(1+tW^2(t)\big) e^{\exp C_0tW(t)}.
$$
         Moreover,
    $$
    \|\partial_X v\|_{W^{1,p}}+\|\partial_X b\|_{W^{1,p}}\le C_0\big(1+W(t)\big)e^{CtW(t)}.
    $$
    with
     $$
     W(t)\triangleq \|\nabla v\|_{L^\infty_tL^\infty}+\|\nabla b\|_{L^\infty_tL^\infty}.
     $$

     \end{prop}
     \begin{proof}
     
      Denote by $\mathcal{L}\triangleq \nabla^\perp\Delta^{-1},$ then from Biot-Savart law we can easily check that 
     $$
     \partial_X v=\mathcal{L}(\partial_X\omega)-[\mathcal{L}, \partial_X]\omega.
     $$
     To estimate the first term of the right-hand side we combine the dyadic partition of the unity  with  Bernstein inequality leading   for $p\in [1,\infty]$ to 
\begin{eqnarray*}
\|\mathcal{L}\partial_X \omega\|_{C^{1-\frac2p}}&\lesssim&\|\Delta_{-1}\nabla^\perp\Delta^{-1}\partial_X \omega\|_{L^\infty}+\|\partial_X \omega\|_{C^{-\frac2p}}\\
&\lesssim&\|\nabla^\perp\Delta^{-1}\partial_X \omega\|_{L^p}+\|\partial_X \omega\|_{C^{-\frac2p}}.
\end{eqnarray*}
According to Lemma \ref{lempro}  the vector field $X_t$ remains solenoidal   and as Riesz transforms are continuous over $L^p$ for  $p\in ]1,\infty[,$ then 
\begin{eqnarray*}
\|\nabla^\perp\Delta^{-1}\partial_X \omega\|_{L^p}&=&\|\nabla^\perp\Delta^{-1}\hbox{div }(X \omega)\|_{L^p}\\
&\lesssim&\|X \omega\|_{L^p}\\
&\lesssim&\|X\|_{L^\infty}\|\omega\|_{L^p}.
\end{eqnarray*}
By the virtue of Lemma \ref{An1}, one obtains
$$
\big\|[\mathcal{L}, \partial_X]\omega\big\|_{C^{1-\frac2p}}\lesssim \|X\|_{C^{1-\frac2p}}\|\omega\|_{L^\infty\cap L^p}.
$$
Putting together the preceding estimates yields
\begin{equation}\label{mohim1}
\|\partial_X v\|_{C^{1-\frac2p}}\lesssim \|\partial_X \omega\|_{C^{{-\frac2p}}}+\|X\|_{C^{1-\frac2p}}\|\omega\|_{L^\infty\cap L^p}.
\end{equation}
Next we shall estimate $\|X(t)\|_{C^{1-\frac2p}}$. For this purpose we use the persistence result of \mbox{Proposition \ref{l555}} which gives for $p\in(2,\infty)$
\begin{equation}\label{dina12}
\|X(t)\|_{C^{1-\frac2p}}\le Ce^{CV(t)}\Big(\|X_0\|_{C^{1-\frac2p}}+\int_0^te^{-CV(\tau)} \|\partial_X v(\tau)\|_{C^{1-\frac2p}}d\tau\Big),
\end{equation}
with $V(t)=\|\nabla v\|_{L^1_t L^\infty}.$ Set $f(t)= e^{-CV(t)}\|\partial_X v\|_{C^{1-\frac2p}},$ then
$$
f(t)\lesssim  \|\partial_X \omega(t)\|_{C^{{-\frac2p}}}+ \|X_0\|_{C^{1-\frac2p}}\|\omega(t)\|_{L^p\cap L^\infty}+\|\omega(t)\|_{L^p\cap L^\infty}\int_0^t f(\tau)d\tau.
$$
This gives in view of Gronwall inequality,
$$
f(t)\lesssim \Big(\|\partial_X \omega\|_{L^\infty_tC^{-\frac2p}}+\|X_0\|_{C^{1-\frac2p}}\|\omega\|_{L^\infty_t(L^p\cap L^\infty)} \Big)e^{Ct\|\omega\|_{L^\infty_t(L^p\cap L^\infty)}}.
$$
Set $W(t)\triangleq \|\nabla v\|_{L^\infty_tL^\infty}+\|\nabla b\|_{L^\infty_tL^\infty}$ then we obtain
from Proposition \ref{a priori}
\begin{eqnarray*}
\|\omega\|_{L^\infty_t(L^p\cap L^\infty)}&\le& C_0 e^{Ct W(t)}+C_0 t W^2(t).
\end{eqnarray*}
Therefore we get by restricting $t\in [0,1]$
\begin{equation*}
\|\partial_X v(t)\|_{C^{1-\frac2p}}\lesssim \Big(\|\partial_X \omega(t)\|_{C^{{-\frac2p}}}+C_0+C_0 tW^2(t)\Big) e^{\exp{C_0t W(t)}}.
\end{equation*}
Performing the same analysis  for the magnetic field we get
\begin{equation*}
\|\partial_X b(t)\|_{C^{1-\frac2p}}\lesssim \Big(\|\partial_X j(t)\|_{C^{{-\frac2p}}}+C_0+C_0 tW^2(t)\Big) e^{\exp{C_0t W(t)}}
\end{equation*}
and consequently
\begin{equation}\label{hmi12}
\|(\partial_X v,\partial_X b)(t)\|_{C^{1-\frac2p}}\lesssim \Big(\|(\partial_X \omega,\partial_X j)(t)\|_{C^{{-\frac2p}}}+C_0+C_0 tW^2(t)\Big) e^{\exp{C_0t W(t)}}.
\end{equation}

To estimate the co-normal regularity of $\omega$ and $j$ we shall first write  down  the  equations of $\partial_X \omega$ and $\partial_X j$.
 According to Proposition \ref{com34} and using the equations \eqref{E} we get
  \begin{equation}
 \left\{ 
\begin{array}{ll} 
D_t\partial_X\omega=\partial_X(b\cdot\nabla j)\\
D_t\partial_Xj=\partial_X(b\cdot\nabla \omega)+2\partial_X(\partial_1b\cdot\nabla v^2-\partial_2b\cdot\nabla v^1).
\end{array} \right.    
     \end{equation}
          From the relation \eqref{comm1}, we get
     $$
     \partial_X\partial_b -\partial_b\partial_X=\partial_{\partial_X b-\partial_b X}.
     $$
     Consequently,
     \begin{equation}\label{com57}
 \left\{ 
\begin{array}{ll} 
D_t\partial_X\omega=b\cdot\nabla\partial_X j+\partial_{\partial_X b-\partial_b X} \,j \\
D_t\partial_Xj=b\cdot\nabla \partial_X\omega+\partial_{\partial_X b-\partial_b X}\,\omega+2\partial_X(\partial_1b\cdot\nabla v^2-\partial_2b\cdot\nabla v^1)
\end{array} \right.    
     \end{equation}
    Set $Y_t=\partial_{X_t} b_t-\partial_{b_t} X_t$ then we can easily check that from our choice we obtain at time zero
     $$
     Y_0=0.
     $$
     Now according to Lemma \ref{com27}, we get
     $$
     Y_t=0,\quad\forall t\in[0,T].
     $$
     Therefore equations \eqref{com57} become
      \begin{equation*}
 \left\{ 
\begin{array}{ll} 
D_t\partial_X\omega=b\cdot\nabla\partial_X j \\
D_t\partial_Xj=b\cdot\nabla \partial_X\omega+2\partial_X(\partial_1b\cdot\nabla v^2-\partial_2b\cdot\nabla v^1).
\end{array} \right.    
     \end{equation*}
     The estimate of the last term of the second equation  in the H\"older space with negative index $C^{{-\frac2p}}$ is quite difficult due to the fact that the product  $L^\infty\times C^{{-\frac2p}}$ is not contained in $C^{{-\frac2p}}.$ To avoid  this technical difficulty we shall replace the space $C^{{-\frac2p}}$ by   Lebesgue space $L^p$ which scales at the same level and satisfies  $L^{p}\hookrightarrow C^{1-\frac2p}$. Using Proposition \ref{transport} one gets
\begin{equation}\label{pro567}
\|(\partial_X\omega,\partial_X j)(t)\|_{L^{p}}\le C_0+ C\int_0^t\Big(\|\partial_X\nabla b(\tau)\|_{L^{p}}\|\nabla v(\tau)\|_{L^\infty}+\|\partial_X\nabla v(\tau)\|_{L^p}\|\nabla b(\tau)\|_{L^\infty}\Big)d\tau.
\end{equation}
From Biot-Savart law we have easily
$$
\|\partial_X\nabla b\|_{L^{p}}\le \sum_{i,k=1}^2\|\partial_X\mathcal{R}_{i,k}j\|_{L^{p}}\quad\hbox{and}\quad \|\partial_X\nabla v\|_{L^p}\le \sum_{i,k=1}^2\|\partial_X\mathcal{R}_{i,k}\omega\|_{L^p}
$$
where $\mathcal{R}_{i,k}=\partial_i\partial_j\Delta^{-1}.$ Now we shall combine the identity
$$
\partial_X\mathcal{R}_{i,k}\omega=\mathcal{R}_{i,k}\big(\partial_X\omega\big)-\big[\mathcal{R}_{i,k}, \partial_X\big]\omega  
$$
together with the continuity of Riesz transforms on the $L^p$ spaces with $p\in(1,\infty)$ leading finally to 
$$
\|\partial_X\mathcal{R}_{i,k}\omega\|_{L^p}\lesssim \|\partial_X\omega\|_{L^p}+\big\|[\mathcal{R}_{i,k}, \partial_X]\omega\|_{L^p}.
$$
At this stage we shall use  Calder\'on's estimate, see Lemma \ref{cald},
\begin{equation}\label{tit11}
\big\|[\mathcal{R}_{i,j}, \partial_X]\omega\|_{L^p}\le C\|\nabla X\|_{L^\infty}\|\omega\|_{L^p}.
\end{equation}
It is not at all obvious how to bound the Lipschitz norm of the vector field $X$ from its evolution equation due to the low regularity of $v$ and as we shall see  its specific  structure will be of great importance to reach this target. Indeed, we know that at time zero the magnetic field is given by $b_0=\nabla^\perp\{G(\varphi_0)\}.$ Thus it follows from Lemma \ref{lempro} that
\begin{equation}\label{eqde}
b_t=\nabla^\perp\{G(\varphi(t))\}\quad\hbox{and}\quad X_t=\nabla^\perp\varphi_t
\end{equation}
with $\varphi_t$ the unique solution of the transport equation 
$$
D_t\varphi=0,\quad\varphi(0)=\varphi_0.
$$
Therefore we get the relation   $b_t=G^\prime(\varphi_t) X_t$ and thus differentiating with respect to the spatial variable we obtain 
$$
\|\nabla X_t G^\prime(\varphi)\|_{L^\infty}\leq \|\nabla b_t\|_{L^\infty}+\|X_t\|_{L^\infty}^2\|G^{\prime\prime}\|_{L^\infty}.
$$
By the assumptions $G^{\prime\prime}$ is bounded and $|G^\prime|$ is bounded below by a  positive constant which imply 
$$
\|\nabla X_t\|_{L^\infty}\lesssim \|\nabla b_t\|_{L^\infty}+\|X_t\|_{L^\infty}^2.
$$
Coming back to the equation \eqref{tran12} and using the maximum principle and Gronwall inequality  we get easily
\begin{eqnarray}\label{inf1}
\nonumber\|X(t)\|_{L^\infty}&\le& \|X_0\|_{L^\infty}+\int_0^t\|X(\tau)\|_{L^\infty}\|\nabla v(\tau)\|_{L^\infty}d\tau\\
&\le& C_0 e^{CV(t)}.
\end{eqnarray}
Hence
\begin{eqnarray}\label{vor7}
\|X_t\|_{\textnormal{Lip}}\lesssim \|\nabla b_t\|_{L^\infty}+C_0 e^{Ct W(t)}.
\end{eqnarray}
Plugging this estimate into \eqref{tit11} and using  Proposition \ref{a priori} one obtains
\begin{eqnarray*}
\big\|[\mathcal{R}_{i,j}, \partial_X]\omega(t)\|_{L^p}&\le& C_0\big(W(t)+e^{CtW(t)}\big)e^{CtW(t)}\\
&\le& C_0\big(W(t)+1\big) e^{CtW(t)}.
\end{eqnarray*}
Thus we deduce
\begin{equation}\label{eq2}
\|\partial_X\mathcal{R}_{i,k}\omega(t)\|_{L^p}\lesssim C_0\big(W(t)+1\big) e^{CtW(t)}+\|\partial_X\omega(t)\|_{L^p}.
\end{equation}
Similarly we get for the current density the estimate, 
$$
\|\partial_X\mathcal{R}_{i,k}\, j(t)\|_{L^p}\lesssim C_0\big(W(t)+1\big) e^{CtW(t)}+\|\partial_X\, j(t)\|_{L^p}.
$$
Inserting these estimates into \eqref{pro567} yields in view of Gronwall inequality
\begin{eqnarray}\label{fin1}
\nonumber\|(\partial_X\omega,\partial_X j)(t)\|_{L^p}&\le& C_0\big(tW^2(t)+tW(t)\big) e^{CtW(t)}+\int_0^tW(\tau)\|(\partial_X\omega,\partial_X j)(\tau)\|_{L^p}\\
&\le&C_0\big(1+tW^2(t)\big) e^{CtW(t)}.
\end{eqnarray}
Combining this estimate with \eqref{hmi12} gives  for $t\in [0,1]$
$$
\| (\partial_X v,\partial_X b)(t)\|_{C^{1-\frac2p}}\le C_0\big(1+tW^2(t)\big) e^{\exp C_0tW(t)}.
$$
Putting this estimate in  \eqref{dina12} one gets for $t\in [0,1],$
\begin{equation}\label{pr24}
\|X(t)\|_{C^{1-\frac2p}}\le C_0\big(1+tW^2(t)\big) e^{\exp C_0tW(t)}.
\end{equation}
Now we shall estimate the co-normal regularity of   $\partial_X v$ and $\partial_X b$ in $W^{1,p}.$  First it is easily seen  that for $p\in(1,\infty)$
     \begin{eqnarray*}
     \|\partial_X v\|_{L^p}&\le& \|X\|_{L^\infty}\|\nabla v\|_{L^p}\\
     &\lesssim& \|X\|_{L^\infty}\|\omega\|_{L^p}.
     \end{eqnarray*}
     Denote by $\mathcal{L}\triangleq \nabla^\perp\Delta^{-1},$ then from Biot-Savart law, 
     
     \begin{eqnarray*}
     \partial_i\partial_X v&=&\partial_iX\cdot\nabla v+\partial_X\partial_i\mathcal{L}\omega \\
     &=&\partial_iX\cdot\nabla v+\partial_i\mathcal{L}\partial_X\omega-\big[\partial_i\mathcal{L},\partial_X   \big]\omega.
     \end{eqnarray*}
    Since Riesz transform $\partial_i\mathcal{L}$ is continuous over $L^p$ then we deduce
    $$
   \| \partial_iX\cdot\nabla v+\partial_i\mathcal{L}\partial_X\omega\|_{L^p}\lesssim \|\nabla X\|_{L^\infty}\|\omega\|_{L^p}+\|\partial_X\omega\|_{L^p}.
    $$

Using Lemma \ref{cald}, we get
$$
\big\|[\partial_i\mathcal{L}, \partial_X]\omega\big\|_{L^p}\lesssim \|\nabla X\|_{L^\infty}\|\omega\|_{ L^p}.
$$
Putting together the preceding estimates implies
\begin{equation}\label{mohim001}
\|\partial_X v\|_{W^{1,p}}\lesssim \|\partial_X \omega\|_{L^p}+\|X\|_{\textnormal{Lip}}\|\omega\|_{ L^p}.
\end{equation}
Performing the same computations for the magnetic field gives  the estimate
\begin{equation}\label{mohim0012}
\|\partial_X b\|_{W^{1,p}}\lesssim \|\partial_X j\|_{L^p}+\|X\|_{\textnormal{Lip}}\|j\|_{ L^p}.
\end{equation}
Combining  Proposition \ref{a priori}
  with the estimates \eqref{mohim001}, \eqref{mohim0012}, \eqref{vor7} and \eqref{fin1} gives
 \begin{eqnarray}\label{hmi2112}
\nonumber \|\partial_X v(t)\|_{W^{1,p}}+\|\partial_X b(t)\|_{W^{1,p}}&\le& C_0\Big(1+tW^2(t)+\|\nabla b(t)\|_{L^\infty}\Big)e^{CtW(t)}\\
&\le&  C_0\big(1+W(t)\big)e^{CtW(t)}.
\end{eqnarray}

\end{proof}
\subsection{Persistence of the  regularity far from the boundary}
We have seen in the previous section how to propagate the co-normal regularity of the solution using in a crucial way the special structure of the magnetic field which should be tangential to the boundary. However the vector field $X_0$ is singular at some points far from the boundary and thus we cannot recover the regularity everywhere. The idea to follow is simple: to track the regularity  far from the singular set we can use somehow the hyperbolic structure of the equations through the classical  principle of  finite speed propagation of the smooth part.   Even though the equations are not local, we shall  prove that the singular set does not affect for small time the smooth part of the solution. Before giving more details we need to recall the following notations:
$$
\mathcal{Z}^\delta_{X_0}\triangleq\big\{ x\in\RR^2,\quad |X_0(x)|\le \delta\big\},\quad \Sigma_{\textnormal{sing}}^{1-\frac2p}\triangleq \Sigma_{\textnormal{sing}}^{1-\frac2p}(\omega_0)\cup \Sigma_{\textnormal{sing}}^{1-\frac2p}(j_0)
$$
and
$$
W(t)\triangleq \sup_{\tau\in[0,t]}\big(\| \nabla v(\tau)\|_{L^\infty}+\| \nabla b(\tau)\|_{L^\infty}\big),\quad V(t)\triangleq \int_0^t\big(\|\nabla v(\tau)\|_{L^\infty}+\|\nabla b(\tau)\|_{L^\infty}\big)d\tau.
$$
\begin{prop}
Let $v_0$ and $b_0$ be  two divergence-free vector fields satisfying the assumptions of Theorem \ref{thm2}.
Let $\psi$ denote the flow associated to the velocity $v$. Then there exists a function $(t,x)\mapsto \chi(t,x)\in [0,1],$ taking $1$ in a neighborhood of $\psi(t,\mathcal{Z}^\delta)$ and vanishing  around $\psi(t,\Sigma_{\textnormal{sing}}^{1-\frac2p})$ such that: for any  $T$ satisfying
$$
C_0T W^4(T) e^{\exp{C_0TW(T)}}\leq1
$$
we get
$$
\forall t\in[0,T],\quad \|\chi(t)\omega(t)\|_{C^{1-\frac2p}}+\|\chi(t) j(t)\|_{C^{1-\frac2p}}\le C_0(1+tW^2(t)) e^{\exp{C_0t W(t)}}.
$$
\end{prop}

\begin{proof}

From the assumption \eqref{separa} and  the compactness of the singular set $\Sigma_{\textnormal{sing}}^{1-\frac2p}$ we can easily prove the existence of small $\delta>0$ depending particularly on $\|\nabla\varphi_0\|_{L^\infty}$ such that  
$$\textnormal{dist}\big(\varphi_0(\mathcal{Z}^\delta_{X_0}),\varphi_0(\Sigma_{\textnormal{sing}}^{{1-\frac2p},\delta})\big)>0,
$$
where we denote by  
$$
\Sigma_{\textnormal{sing}}^{{1-\frac2p},\delta}\triangleq\Big\{x\in\RR^2; d(x, \Sigma_{\textnormal{sing}}^{1-\frac2p})\le \delta\Big\}.
$$
Using Urysohn Theorem 
 we can construct a smooth function $H:\RR\to[0,1]$ such that
 \begin{equation*}
H(\theta) =\left\{ 
\begin{array}{ll} 
1,\quad\hbox{if}\quad \theta\in\varphi_0(\mathcal{Z}^\delta_{X_0}),\\
0,\quad\hbox{if}\quad \theta\in \varphi_0(\Sigma_{\textnormal{sing}}^{{1-\frac2p},\delta}).
\end{array} \right.    
     \end{equation*}
     An explicit formula for this function in the  Lipschitz class   is given by 
     $$
     H(\theta)=\frac{\textnormal{dist}(\theta, A)}{\textnormal{dist}(\theta, A)+\textnormal{dist}(\theta, B)}, \quad A\triangleq\varphi_0(\Sigma_{\textnormal{sing}}^{{1-\frac2p},\delta}),\quad  B\triangleq \varphi_0(\mathcal{Z}^\delta_{X_0}).
     $$
     We point out that by enlarging a little bit $A$ and $B$ and  following a  smoothing procedure we can construct $H$ in $C^\infty$ class with bounded derivatives. 
     Now introduce the new function 
     \begin{equation}\label{def9g}
     \chi_0(x)=H(\varphi_0(x)).
     \end{equation}
     Since $\varphi_0$ belongs to $W^{2,\infty}$ and $H$ is very smooth we get $\chi_0\in W^{2,\infty}.$  It is easy to check that  this function satisfies the following properties:
     \begin{equation*}
\chi_0(x) =\left\{ 
\begin{array}{ll} 
1,\quad\hbox{if}\quad x\in \mathcal{Z}^\delta_{X_0},\\
0,\quad\hbox{if}\quad x\in\Sigma_{\textnormal{sing}}^{{1-\frac2p},\delta}.
\end{array} \right.       
\end{equation*}
It is clear that
\begin{equation}\label{Eq3}
\forall x\in \RR^2,\quad 1-\chi_0(x)\neq0\Longrightarrow |X_0(x)|>\delta.
\end{equation}
Moreover according once again to the assumptions of Theorem \ref{thm2}, the functions $ \chi_0\omega_0$ and $\chi_0 j_0$ belongs to the space $C^{1-\frac2p}.$
Let $\varphi$ be the solution of the transport equation 
     \begin{equation}\label{Eqs1}
\left\{ 
\begin{array}{ll} 
D_t\varphi=0,\\
\varphi(0,x)=\varphi_0(x).
\end{array} \right.    
     \end{equation}
then by Lemma \ref{lempro}
$$X(t,x)=\nabla^\perp\varphi(t,x). 
$$
Define the cut-off function 
\begin{equation}\label{hamil1}
 \chi(t,x)=H(\varphi(t,x)).
\end{equation}
  Then  it is easy seen that,  \begin{equation}\label{Eq4}
\left\{ 
\begin{array}{ll} 
D_t\chi=0,\\
\chi(0,x)=\chi_0(x).
\end{array} \right. 
     \end{equation}
     We shall now prove the following  assertion,
     \begin{equation}\label{Eq5}
\forall x\in\RR^2,\quad 1-\chi(t,x)\neq0\Longrightarrow |X(t,x)|\geq\delta e^{-CV(t)}.
\end{equation}
Indeed, set $Y(t,x)=X(t,\psi(t,x))$ where $\psi$ is the flow associated to the vector field $v.$ Then
\begin{eqnarray*}
\partial_tY(t,x)&=&(D_tX)(t,\psi(t,x))\\
&=&Y(t,x)\cdot\{(\nabla v)(t,\psi(t,x))\}
\end{eqnarray*}
 and  we get by Gronwall inequality
 $$
 |Y(t,x)|\le|X_0(x)| e^{C\int_0^t\|\nabla v(\tau)\|_{L^\infty}d\tau}.
 $$ 
 Combining this estimate with the reversibility of the equation gives 
 \begin{eqnarray*}
 |X_0(x)|&\le&  |Y(t,x)|e^{C\int_0^t\|\nabla v(\tau)\|_{L^\infty}d\tau}, \end{eqnarray*} 
which means that
 $$
 |{X}(t,x)|\geq |X_0(\psi^{-1}(t,x))|e^{-CV(t)}.
 $$
 Consequently,  by \eqref{Eq3}  one gets
 \begin{equation*}
\forall x\in\RR^2,\quad 1-\chi_0(\psi^{-1}(t,x))\neq0\Longrightarrow |X(t,x)|\geq\delta\,e^{-CV(t)}.
\end{equation*}
Since $\chi(t,x)=\chi_0(\psi^{-1}(t,x)$ then the proof of \eqref{Eq5} is now complete.
The next step is to estimate the regularity of the solutions far from the boundary. To do so we start with the following   notations
$$
f^-(t,x)=f(t,x) \chi(t,x),\quad f^+(t,x)=f(t,x) (1-\chi(t,x)).
$$
Combining \eqref{E} and \eqref{Eq4}, we find that $\omega^-$ satisfies the equation
$$
D_t\omega^-=\chi b\cdot\nabla j=b\cdot\nabla j^--j\, b\cdot\nabla\chi.
$$
According to  \eqref{eqde} we get
$$
b(t,x)=\nabla^\perp\{G(\varphi(t,x))\}.
$$  
which yields in view of  \eqref{hamil1} to
\begin{eqnarray*}
{b}(t,x)\cdot\nabla\chi(t,x)&=&0.
\end{eqnarray*}
 Therefore  the equation of $\omega^-$ becomes
\begin{equation}\label{hmi9}
D_t\omega^-=b\cdot\nabla j^-.
\end{equation}
By the same way we can establish that
\begin{equation}\label{hmi6}
D_t j^-=b\cdot\nabla j^-+2\chi\big(\partial_1b\cdot\nabla v^2-\partial_2 b\cdot\nabla v^1\big). 
\end{equation}
  Since $b=b^++b^-$ then    the last term can be decomposed  as follows
     
    \begin{eqnarray}\label{hmid7}
   \nonumber \chi\big(\partial_1b\cdot\nabla v^2-\partial_2 b\cdot\nabla v^1\big)&=&\chi\big(\partial_1b^+\cdot\nabla v^2-\partial_2 b^+\cdot\nabla v^1\big)\\
   \nonumber &+&\chi\big(\partial_1b^-\cdot\nabla v^2-\partial_2 b^-\cdot\nabla v^1\big)\\
    &\triangleq&\hbox{I}+\hbox{II}.
   \end{eqnarray}  
   Straightforward  computations give for the first term
    \begin{eqnarray*}
     \hbox{I}&=&\chi(1-\chi)\big(\partial_1b\cdot\nabla v^2-\partial_2 b\cdot\nabla v^1\big)\\
     &-&\chi\big(\partial_1\chi\,b\cdot\nabla v^2-\partial_2\chi\, b\cdot\nabla v^1\big)\\
     &\triangleq& \hbox{I}_1+\hbox{I}_2.
     \end{eqnarray*}
     To estimate the term $\hbox{I}_1$ we shall use the identity \eqref{Iden1}, 
     
       \begin{eqnarray}\label{Iden3}
     \nonumber  \hbox{I}_1&=&2\chi\frac{1-\chi}{|X|^2}\Big\{\partial_X b^1\,\partial_X v^2-\partial_X b^2\,\partial_X v^1\Big\}\\
  &+&\frac{1-\chi}{|X|^2}\Big\{j^-\, X\cdot\partial_X v -\omega^-\, X\cdot \partial_X b\Big\}.
     \end{eqnarray}
     
     Using  \eqref{Eq5} and the algebra structure of  $C^\varepsilon$ for $\EE>0,$ we get

     \begin{eqnarray}\label{Eq6}
    \nonumber \|\hbox{I}_1(t)\|_{C^{1-\frac2p}}&\le&Ce^{CV(t)}\big(1+\|\chi(t)\|_{C^{1-\frac2p}}^2\big)\|X(t)\|_{C^{1-\frac2p}}\|\partial_X v(t)\|_{C^{1-\frac2p}}|\partial_X b(t)\|_{C^{1-\frac2p}}\\
     \nonumber&+&Ce^{CV(t)}\|\chi(t)\|_{C^{1-\frac2p}}\big(1+\|X(t)\|_{C^{1-\frac2p}}^2\big)\Big\{\|j^-(t)\|_{C^{1-\frac2p}}\|\partial_X v(t)\|_{C^{1-\frac2p}}\\
     &&\qquad\qquad\qquad \qquad\qquad\qquad\qquad+\|\omega^-(t)\|_{C^{1-\frac2p}}\|\partial_X b(t)\|_{C^{1-\frac2p}}\Big\}.
          \end{eqnarray}
          
 To estimate $\|\chi\|_{C^{1-\frac2p}}$  , we apply  Proposition \ref{l555} to the equation  \eqref{Eq4},
 
\begin{eqnarray}\label{ham2}
\nonumber \|\chi(t)\|_{C^{1-\frac2p}}&\le& 
  C  \|\chi_0\|_{C^{1-\frac2p}}e^{CV(t)}\\
  &\le& C_0 e^{CtW(t)}.
 \end{eqnarray}
 
Combining Proposition \ref{prop1} with  \eqref{Eq6} and \eqref{ham2}  we get for $t\in[0,1]$
\begin{eqnarray}\label{hmid4}
\nonumber\|\hbox{I}_1(t)\|_{C^{1-\frac2p}}&\le& C_0(1+tW^2(t))e^{\exp{C_0 tW(t)}} \Big(1+\|j^-(t)\|_{C^{1-\frac2p}}+\|\omega^-(t)\|_{C^{1-\frac2p}}\Big)\\
&\le&C_0 e^{\exp{C_0 tW(t)}} \big(1+W(t)\big)\Big(1+\|j^-(t)\|_{C^{1-\frac2p}}+\|\omega^-(t)\|_{C^{1-\frac2p}}\Big).
\end{eqnarray}

The term $\hbox{I}_2$ can be estimated as follows,
          \begin{eqnarray}\label{Eq7}
     \|\hbox{I}_2(t)\|_{C^{1-\frac2p}}&\le&C\|\nabla\chi(t)\|_{C^{1-\frac2p}}\|b(t)\|_{C^{1-\frac2p}}\|\chi\nabla v(t)\|_{C^{1-\frac2p}}.
          \end{eqnarray}
For the last term of the right-hand side we write,
\begin{eqnarray*}
\chi\partial_i v&=&\chi\partial_i\nabla^\perp\Delta^{-1}\omega\\
&=&\partial_i\nabla^\perp\Delta^{-1}(\omega^-)-\Big[\partial_i\nabla^\perp\Delta^{-1},\chi\Big]\omega.
\end{eqnarray*}
The first term can be treated by using Bernstein inequality  leading  to
\begin{eqnarray*}
\|\partial_i\nabla^\perp\Delta^{-1}(\omega^-)\|_{C^{1-\frac2p}}&\lesssim&\|\Delta_{-1}\partial_i\nabla^\perp\Delta^{-1}(\omega^-)\|_{L^\infty}+\|\omega^-\|_{C^{1-\frac2p}}\\&\lesssim&\|\omega^-\|_{L^p}+\|\omega^-\|_{C^{1-\frac2p}}\\
&\lesssim&\|\omega\|_{L^p}+\|\omega^-\|_{C^{1-\frac2p}}.
\end{eqnarray*}
Using Proposition \ref{a priori} we find
$$
\|\partial_i\nabla^\perp\Delta^{-1}\omega^-(t)\|_{C^{1-\frac2p}}\le C_0e^{CtW(t)}+\|\omega^-(t)\|_{C^{1-\frac2p}}.
$$
As to the commutator term we use Lemma \ref{cald},
$$
\Big\|\big[\partial_i\nabla^\perp\Delta^{-1},\chi\big]\omega\Big\|_{C^{1-\frac2p}}\lesssim\|\chi\|_{\textnormal{Lip}}\|\omega\|_{L^p}.
$$
Since $\chi$ is transported by the flow then
\begin{eqnarray*}
\|\chi(t)\|_{\textnormal{Lip}}&\le& C\| \chi_0\|_{\textnormal{Lip}}e^{C\|\nabla v\|_{L^1_tL^\infty}}\\
&\le& C_0e^{CtW(t)}.
\end{eqnarray*}
Hence we find using once again Proposition \ref{a priori}
$$
\Big\|\big[\partial_i\nabla^\perp\Delta^{-1},\chi\big]\omega(t)\Big\|_{C^{1-\frac2p}}\le C_0e^{CtW(t)}.
$$
Putting together the preceding estimates  gives
\begin{equation}\label{ham243}
\|\chi\nabla v(t)\|_{C^{1-\frac2p}}\le C_0e^{CtW(t)}+\|\omega^-(t)\|_{C^{1-\frac2p}}.
\end{equation} 
Coming back to \eqref{Eq7}, then it remains to estimate $\|b(t)\|_{C^{1-\frac2p}}$ and $\|\nabla\chi(t)\|_{C^{1-\frac2p}}.$ The first term is estimated as follows,
\begin{eqnarray}\label{Eq8}
\nonumber\|b(t)\|_{C^{1-\frac2p}}&\lesssim& \| b(t)\|_{L^\infty}+ W(t)\\
\nonumber&\lesssim& \|b_0\|_{L^\infty} e^{C\|\nabla v\|_{L^1_t L^\infty}}+W(t)\\
&\lesssim& C_0e^{CtW(t)}+W(t).
\end{eqnarray}
Concerning the second one $\|\nabla \chi(t)\|_{C^{1-\frac2p}}$ recall that $\nabla^\perp\varphi(t)=X(t)$ and $\chi(t,x)=H(\varphi(t,x)$ which imply
         \begin{eqnarray}
          \nonumber\|\nabla^\perp\chi(t)\|_{C^{1-\frac2p}}&=& \|H^\prime(\varphi(t))\,X(t)\|_{C^{1-\frac2p}}\\
          \nonumber&\le&\|H^\prime(\varphi(t))\|_{C^{1-\frac2p}}\|X(t)\|_{C^{1-\frac2p}}.
                 \end{eqnarray}
         Now we use the classical composition law
         $$
         \|H^\prime(\varphi(t))\|_{C^{1-\frac2p}}\lesssim\|H^\prime\|_{W^{1,\infty}}\big(1+\|\varphi(t)\|_{C^{1-\frac2p}}\big)
         $$      
         which gives according to Proposition \ref{a priori} and Proposition \ref{prop1} that for $t\in[0,1]$, 
                 \begin{eqnarray}\label{Eq10}
                 \|\nabla^\perp \chi(t)\|_{C^{1-\frac2p}}&\le& C_0\big(1+tW^2(t)\big)e^{\exp{C_0tW(t)}}\\
                \nonumber &\le&C_0 (1+W(t))e^{\exp{C_0tW(t)}}.
                 \end{eqnarray}
          Putting together \eqref{Eq7},\eqref{ham243}, \eqref{Eq8} and \eqref{Eq10} we obtain
         \begin{equation}\label{Eqd}
         \|\hbox{I}_2(t)\|_{C^{1-\frac2p}}\le  C_0 \big[1+W^2(t)\big] e^{\exp{C_0tW(t)}}\big(1+\|\omega^-(t)\|_{C^{1-\frac2p}}\big)
         \end{equation}
                          
                 Combining this estimate with \eqref{hmid4} we get
                  \begin{equation}\label{zz1}
                 \|\hbox{I}(t)\|_{C^{1-\frac2p}}\le  C_0\big[1+ W^2(t)\big] e^{\exp{C_0tW(t)}}\Big(1+\|\omega^-(t)\|_{C^{1-\frac2p}}+\|j^-(t)\|_{C^{1-\frac2p}}\Big).
                 \end{equation}
                 Coming back to the estimate of the second term  $\hbox{II}$ of \eqref{hmid7}. From the algebra structure of $C^{1-\frac2p}$,
                 $$
                   \|\hbox{II}(t)\|_{C^{1-\frac2p}}\lesssim \|\nabla b^-(t)\|_{C^{1-\frac2p}}\|\chi\nabla v(t)\|_{C^{1-\frac2p}}.
                 $$
                 Therefore we get according to \eqref{ham243}
                 $$
                   \|\hbox{II}(t)\|_{C^{1-\frac2p}}\lesssim \Big(C_0 e^{CtW(t)}+\|\omega^-(t)\|_{C^{1-\frac2p}}\Big)\|\nabla b^-(t)\|_{C^{1-\frac2p}}.
                                    $$
         The last term will be estimated as follows,
         \begin{eqnarray*}
         \|\partial_i b^-\|_{C^{1-\frac2p}}&\le& \|\chi\partial_i b\|_{C^{1-\frac2p}}+\|b\partial_i\chi\|_{C^{1-\frac2p}}\\
         &\le& \|\chi\partial_i b\|_{C^{1-\frac2p}}+\|b\|_{C^{1-\frac2p}}\|\nabla\chi\|_{C^{1-\frac2p}}
         \end{eqnarray*}
                     Using \eqref{Eq8} and  \eqref{Eq10}  we obtain
                    \begin{eqnarray*}
         \|\partial_i b^-(t)\|_{C^{1-\frac2p}}\le  \|\chi(t)\partial_i b(t)\|_{C^{1-\frac2p}}+C_0\big[1+W^2(t)\big] e^{\exp{C_0tW(t)}}.
         \end{eqnarray*}
       Concerning the estimate of the first term of the right-hand side we imitate the same computation of \eqref {ham243}       
       $$
       \|\chi\partial_i b\|_{C^{1-\frac2p}}\le  C_0e^{CtW(t)}+\|j^-\|_{C^{1-\frac2p}}.
       $$ 
              It follows that
              $$
               \|\partial_i b^-\|_{C^{1-\frac2p}}\le  \|j^-\|_{C^{1-\frac2p}}+C_0\big[1+W^2(t)\big] e^{\exp{C_0tW(t)}}.
              $$
              Putting together the previous estimates
              $$
               \|\hbox{II}\|_{C^{1-\frac2p}}\le C_0\Big(1+\|\omega^-(t)\|_{C^{1-\frac2p}}\Big)\Big(1+ \|j^-\|_{C^{1-\frac2p}}\Big)\big[1+W^2(t)\big] e^{\exp{C_0tW(t)}}.
              $$
              Inserting this estimate and  \eqref{zz1} into \eqref{hmid7} gives
                      $$
           \|\chi(t)\big(\partial_1b\cdot\nabla v^2-\partial_2 b\cdot\nabla v^1\big)(t)\|_{C^{1-\frac2p}}\le C_0\Big(1+\|\omega^-(t)\|_{C^{1-\frac2p}}\Big)\Big( \|j^-\|_{C^{1-\frac2p}}+1\Big)\big[1+W^2(t)\big]  e^{\exp{C_0tW(t)}}.
                      $$   
           Set 
           $$g(t)\triangleq  \|\omega^-(t)\|_{C^{1-\frac2p}}+\|j^-(t)\|_{C^{1-\frac2p}},
           $$
           then applying Proposition \ref{transport} to the system \eqref{hmi9} and \eqref{hmi6}  we get for $t\in [0,1],$
           $$
         g(t)\le C_0e^{\exp{C_0tW(t)}}\big[1+tW^2(t)\big]+C_0e^{\exp{C_0tW(t)}}\int_0^tg^2(\tau)\big[1+W^2(\tau)\big]d\tau.
           $$
           It follows that for small time $t\in[0,T]$ such that
           \begin{equation}\label{cond1}
           4C_0^2Te^{\exp{C_0TW(T)}}\big[1+W^2(T)\big]^2\le1
           \end{equation}
           we get 
           $$
            g(t)\le 2C_0e^{\exp{C_0tW(t)}}\big[1+tW^2(t)\big].
           $$
           This completes the proof of the proposition. We point out that as a by-product of \eqref{ham243} one obtains 
           
           \begin{equation}\label{ko1}
           \|\chi(t)\nabla v(t)\|_{C^{1-\frac2p}}+ \|\chi(t)\nabla b(t)\|_{C^{1-\frac2p}}\le  C_0e^{\exp{C_0tW(t)}}\big[1+tW^2(t)\big].
           \end{equation}
           \end{proof}
           \subsection{Proof of Theorem \ref{thm2}}
           We shall now discuss the proof of Theorem \ref{thm2}. We first establish the suitable a priori estimates and second we sketch  the principal ingredients for the construction of the solution in our context. We end with the uniqueness part.
           \begin{proof}
           We shall start with the local a priori estimates.
           
           $\bullet$ {\it Local a priori estimates.}
           \vspace{0,3cm}
           
            We assume that the system \eqref{MHD} admits a smooth solution and we wish to find some a priori estimates. The crucial quantities for the persistence of the regularity are   the Lipschitz norms of the velocity and the magnetic field. 
           To estimate the Lipschitz norm of the velocity we shall use \eqref{ko1}. Then under the \mbox{assumption \eqref{cond1}}
           \begin{eqnarray*}
           \|\nabla v(t)\|_{L^\infty}&\le& \|\chi(t)\nabla v(t)\|_{L^\infty}+\|(1-\chi(t))\nabla v(t)\|_{L^\infty}\\
           &\le&\|\chi(t)\nabla v(t)\|_{C^{1-\frac2p}}+\|(1-\chi(t))\nabla v(t)\|_{L^\infty}\\
           &\le&  C_0e^{\exp{C_0tW(t)}}\big[1+tW^2(t)\big]+\|(1-\chi(t))\nabla v(t)\|_{L^\infty}
           \end{eqnarray*}
           To estimate the last term we shall use the identity \eqref{vort},
          \begin{eqnarray*}
           \|(1-\chi(t))\nabla v(t)\|_{L^\infty}&\le& \Big{\|}\frac{1-\chi(t)}{|X(t)|^2}\Big{\|}_{L^\infty}\Big(\|X(t)\|_{L^\infty}\|\partial_X v(t)\|_{L^\infty}+\|X(t)\|_{L^\infty}^2\|\omega(t)\|_{L^\infty}\Big)\\
           &\le&
           \Big{\|}\frac{1-\chi(t)}{|X(t)|^2}\Big{\|}_{L^\infty}\Big(\|X(t)\|_{L^\infty}\|\partial_X v(t)\|_{C^{1-\frac2p}}+\|X(t)\|_{L^\infty}^2\|\omega(t)\|_{L^\infty}\Big).
           \end{eqnarray*}
           Using Proposition \ref{prop1} combined with  \eqref{inf1},  \eqref{Eq5} and Proposition \ref{a priori} 
           \begin{eqnarray*}
             \|(1-\chi(t))\nabla v(t)\|_{L^\infty}&\le& C_0\big[1+tW^2(t)\big]e^{\exp{C_0 tW(t)}}\Big(1+\int_0^t\|\nabla v(\tau)\|_{L^\infty}\|\nabla b(\tau)\|_{L^\infty} d\tau\Big)\\
             &\le& C_0\big[1+tW^2(t)\big]^2e^{\exp{C_0 tW(t)}}.
           \end{eqnarray*}
           Consequently  we obtain
           $$
             \|\nabla v(t)\|_{L^\infty}\le C_0\big[1+tW^2(t)\big]^2e^{\exp{C_0 tW(t)}}.
             $$
      In a similar way we get for the magnetic field
     $$
             \|\nabla b(t)\|_{L^\infty}\le C_0\big[1+tW^2(t)\big]^2e^{\exp{C_0 tW(t)}}.
             $$    
             Thus we find under the assumption \eqref{cond1}: 
             $$
            \forall t\in [0,T],\quad W(t)\le C_0\big[1+tW^2(t)\big]^2e^{\exp{C_0 tW(t)}}.
             $$  
             The goal is to find a suitable time existence $T=T(C_0)>0$ subject to the above constraints. We shall look for small $T$ such that
             $
             W(T)< 2 eC_0. 
             $  This holds true whenever 
             $$
             \big(1+4e^2C_0^2\, T\big)^2e^{\exp{2eC_0^2 T}}< 2e.
             $$

           The existence of such $T$  follows  from the continuity in time of left-hand side and the fact that the previous inequality is strict for $T=0.$ It remains to check the condition \eqref{cond1}. This is true if
           \begin{equation}\label{condp}
          4C_0^2\big[1+4e^2 C_0^2\big]^2 Te^{\exp{ 2eC_0^2T}}\leq1.
           \end{equation}
           To guarantee this last condition we take $T$ sufficiently small. Under this assumption we see from the previous computations in the last sections that $\omega(t), j(t)\in W^p_{X(t)}, \quad\forall t\in [0,T].$ Moreover the vector fields $v, b$ and $X$ belong to $L^\infty([0,T]; W^{1,\infty}).$ To achieve the a priori estimates of Theorem \ref{thm2} it remains to check that $\partial_{X_0}\psi(t)\in L^\infty([0,T]; W^{1,\infty}).$ For this aim we use the identity
           $$
           \partial_{X_0}\psi(t)=X_t\circ\psi(t).
           $$
           It suffices now to use the fact that $X_t$ and $\psi(t)$ belong both to the Lipschitz class $W^{1,\infty}.$
           \vspace{0,2cm}
           
              $\bullet$ {\it Existence and smoothing procedure}. 
              \vspace{0,3cm}

              To justify rigorously the previous a priori estimates and construct a solution as claimed in \mbox{Theorem \ref{thm2}} we start with  smoothing out the initial data as follows
              $$
              v_0^n=v_0\star \eta_n, \quad b_0^n=\nabla^\perp\{G(\varphi_n)\},\quad X_0^n=X_0\star\eta_n, \quad\varphi_n=\varphi\star\eta_n
              $$ 
              where $\eta_n(x)=n^2\eta(nx)$ is a standard mollifier. From the assumptions we can easily check that for any  $n,$ 
              $$
              v_0^n, b_0^n\in C^{1+\alpha}
              $$
              for any $\alpha\in (0,1).$ Consequently  we can apply the classical theory which ensures for each $n$ the existence and the uniqueness  of local   solution $(v^n, b^n)$   defined on some interval $[0, T_n]$ and with values in $C^{1+\alpha}$. We shall prove that $\inf_n T_n\geq T>0$ where $T$ is defined  in  \eqref{condp} but this does not mean that the bounds are uniform in the classical space $C^{1+\alpha}$. The uniformness  in tho space is false but it will be proven in the space of the initial data.  Indeed, it suffices to show that the smooth family $(v_0^n, b_0^n)$ satisfies the assumptions of Theorem \ref{thm2} with uniform bounds with respect to $n.$ First we intend to check the first assumption , that is, 
              $$
              \partial_{X_0^n}\omega_0^n,\quad \partial_{X_0^n} j_0^n\in L^p
              $$
              with uniform bounds. 
              First observe that   the vorticity $\omega_0^n$ of $v_0^n$ is given by $ \omega_0^n=\omega_0\star\rho_n$ and
              $$
               \partial_{X_0^n}\omega_0^n(x)= \partial_{X_0^n-X_0}\omega_0^n(x)+ \partial_{X_0}\omega_0^n(x).
              $$
              The first term can be estimated in a classical way as follows
              \begin{eqnarray*}
              \|\partial_{X_0^n-X_0}\omega_0^n\|_{L^p}&\le& \|X_0^n-X_0\|_{L^\infty}\|\nabla \omega_0^n\|_{L^p}\\
              &\le&\|\nabla X_0\|_{L^\infty} \||\cdot|\eta_n\|_{L^1}\|\omega_0\|_{L^p}\|\nabla \eta_n\|_{L^1}\\
              &\lesssim&\|\nabla X_0\|_{L^\infty} \|\omega_0\|_{L^p}.
              \end{eqnarray*}
          As regards the second term we write
              
              \begin{eqnarray*}
               \partial_{X_0}\omega_0^n(x)&=&\eta_n\star(\partial_{X_0}\omega_0)(x)+n^2\int_{\RR^2}[X_0(x)-X_0(y)]\cdot\nabla_y\omega_0(y)\eta(n(x-y))dy\\
               &\triangleq& \hbox{I}_n+\hbox{II}_n.
              \end{eqnarray*}
        Using the convolution inequalities we obtain
        $$
        \|\hbox{I}_n\|_{L^p}\lesssim \|\partial_{X_0}\omega_0\|_{L^p}.
        $$
        Integration by parts combined with the incompressibility of $X_0$ yields 
        $$
        \hbox{II}_n(x)=n^3\int_{\RR^2}[X_0(x)-X_0(y)]\cdot(\nabla_y\eta)(n(x-y))\omega_0(y)dy.
        $$
      Thus we get
      $$
      |\hbox{II}_n(x)|\le\|\nabla X_0\|_{L^\infty}n^3\int_{\RR^2}|x-y||(\nabla_y\eta)(n(x-y))||\omega_0(y)|dy.
      $$     
      From the classical convolution laws one gets
      $$
      \|\hbox{II}_n\|_{L^p}\le \|\nabla X_0\|_{L^\infty} \||\cdot|\nabla\eta\|_{L^1}\|\omega_0\|_{L^p}.
      $$ 
      This achieves the proof of the first assumption of Theorem \ref{thm2}.
      
      Let us now check the second assumption of this theorem. We shall show that
      $$
      \sup_{n\in \NN^\star}\|(1-\rho)\omega_0^n\|_{C^{1-\frac2p}}+ \sup_{n\in \NN^\star}\|(1-\rho)j_0^n\|_{C^{1-\frac2p}}<\infty.
      $$
      We point out that in the application the function $1-\rho$ is closely related to the function $\chi_0$ introduced in \eqref{def9g} and this latter one belongs to $W^{2,\infty}.$ Thus the function $\rho$ should belong to $W^{2,\infty}$ and not more. 
      We write
      \begin{eqnarray*}
      (1-\rho(x))\omega_0^n(x)&=&\eta_n\star[(1-\rho)\omega_0](x)+\int_{\RR^2}[\rho(x)-\rho(y)]\omega_0(y)\eta_n(x-y)dy\\
      &\triangleq&\mathcal{I}_1(x)+\mathcal{I}_2(x)
      \end{eqnarray*}
      From the classical convolution inequalities one gets
      $$
      \|\mathcal{I}_1\|_{C^{1-\frac2p}}\lesssim \|(1-\rho)\omega_0\|_{C^{1-\frac2p}}.
      $$
      For the second term we claim that
      $$
       \|\mathcal{I}_2\|_{W^{1,\infty}}\lesssim \|\omega_0\|_{L^\infty}.
      $$
      Indeed, the uniform boundedness is easy to get. Concerning the Lipschitz norm we write
      $$
      \nabla \mathcal{I}_2(x)=\nabla \rho(x)\,\eta_n\star\omega_0(x)+ n^3 \int_{\RR^2}[\rho(x)-\rho(y)]\omega_0(y)(\nabla\eta)\big( n(x-y)\big)dy.
    $$
      Consequently we find
        \begin{eqnarray*}
      \|\nabla \mathcal{I}_2\|_{L^\infty}&\le &\|\nabla\rho\|_{L^\infty}\|\omega_0\|_{L^\infty}\|\eta_n\|_{L^1}+\|\nabla\rho\|_{L^\infty}\|\omega_0\|_{L^\infty}\||\cdot|\nabla \eta\|_{L^1}\\
      &\lesssim&\|\omega_0\|_{L^\infty}.
        \end{eqnarray*}
        Concerning the uniform estimate of $\|(1-\rho)j_0^n\|_{C^{1-\frac2p}}$ it suffices to bound $(1-\rho) b_0^n$ in the H\"{o}lder  space $C^{2-\frac2p}.$ For this purpose we write by the definition
        $$
        (1-\rho) b_0^n=G^\prime(\varphi_n) [(1-\rho) \nabla^\perp\varphi_n].
        $$
      Using the algebra structure of $W^{2-\frac2p}$ yields
      $$
      \| (1-\rho) b_0^n\|_{C^{2-\frac2p}}\lesssim \|G^\prime(\varphi_n)\|_{C^{2-\frac2p}}\big\| [(1-\rho) \nabla^\perp\varphi_n]\big\|_{C^{2-\frac2p}}.
      $$  
     From the classical law products one obtains
     $$
     \|G^\prime(\varphi_n)\|_{C^{2-\frac2p}}\lesssim  \|G^\prime\|_{W^{2,\infty}}\|\varphi_n\|_{C^{2-\frac2p}}.
     $$    
     Combining this inequality with the convolution laws
     $$
     \|\varphi_n\|_{C^{2-\frac2p}}\lesssim \|\varphi\|_{C^{2-\frac2p}}
     $$
     allows to get
     $$
     \|G^\prime(\varphi_n)\|_{C^{2-\frac2p}}\lesssim  \|G^\prime\|_{W^{2,\infty}}\|\varphi\|_{C^{2-\frac2p}}.
     $$
     On the other hand we have
     \begin{eqnarray*}
     \big\| [(1-\rho) \nabla^\perp\varphi_n]\big\|_{C^{2-\frac2p}}&\lesssim&  \big\| [(1-\rho) \nabla^\perp\varphi_n]\big\|_{L^\infty}+ \big\| \nabla \rho \nabla^\perp \varphi_n\big\|_{W^{1,\infty}}+ \big\| (1-\rho) \nabla^\perp\nabla \varphi_n\big\|_{C^{1-\frac2p}}\\
     &\lesssim&\|\rho\|_{W^{2,\infty}}\|\varphi_n\|_{W^{2,\infty}}+\big\|( \hbox{Id}-\Delta_{-1})[(1-\rho) \nabla^\perp\nabla \varphi_n]\big\|_{C^{1-\frac2p}}\\
     &\lesssim&\|\rho\|_{W^{2,\infty}}\|\varphi\|_{W^{2,\infty}}+\big\|( \hbox{Id}-\Delta_{-1})[(1-\rho) \nabla^\perp\nabla \varphi_n]\big\|_{C^{1-\frac2p}}.
     \end{eqnarray*}
  As to the last term we transform it into
  \begin{eqnarray*}
   (1-\rho) \nabla^\perp\nabla \varphi_n&=&\{\nabla[(1-\rho) \nabla^\perp \varphi]\}\star\eta_n+\int_{\RR^2}[\rho(y)-\rho(x)]\nabla^\perp \varphi(y)\nabla\eta_n(x-y)dy\\
       &\triangleq& \mathcal{J}_1+\mathcal{J}_2.
  \end{eqnarray*}
  Using once again the convolution inequalities we find
  $$
  \big\|\mathcal{J}_1\big\|_{C^{1-\frac2p}}\lesssim \|(1-\rho) \nabla^\perp \varphi\|_{C^{2-\frac2p}}.
  $$
  For the term $\mathcal{J}_2$ we write
  $$
  \|( \hbox{Id}-\Delta_{-1})\mathcal{J}_2\|_{C^{1-\frac2p}}\lesssim \|\nabla \mathcal{J}_2\|_{L^\infty}.
  $$
 It is easy to check that
 $$
 \partial_i \mathcal{J}_2=(\partial_i\rho\nabla^\perp\varphi)\star\nabla\eta_n-\partial_i\rho(x) \nabla^\perp\varphi\star\nabla\eta_n+\int_{\RR^2}[\rho(y)-\rho(x)]\partial_i\nabla^\perp \varphi(y)\nabla\eta_n(x-y)dy.
 $$ 
  The first two terms of the right-hand side can be estimated as follows
  \begin{eqnarray*}
  \|(\partial_i\rho\nabla^\perp\varphi)\star\nabla\eta_n-\partial_i\rho(x) \nabla^\perp\varphi\star\nabla\eta_n\|_{L^\infty}&\lesssim& \|\nabla(\partial_i\rho\nabla^\perp\varphi)\|_{L^\infty}+\|\nabla\rho\|_{L^\infty}\|\nabla\nabla^\perp\varphi\|_{L^\infty}\\
  &\lesssim&\|\rho\|_{W^{2,\infty}}\|\varphi\|_{W^{2,\infty}}.
  \end{eqnarray*}
 Concerning the last term we write
 \begin{eqnarray*}
 \Big|\int_{\RR^2}[\rho(y)-\rho(x)]\partial_i\nabla^\perp \varphi(y)\nabla\eta_n(x-y)dy
\Big|&\le&\|\nabla\rho\|_{L^\infty} \|\nabla\nabla^\perp\varphi\|_{L^\infty}\int_{\RR^2}|y-x||\nabla\eta_n(x-y)|dy\\
&\lesssim&\|\nabla\rho\|_{L^\infty}\|\varphi\|_{W^{2,\infty}}. 
 \end{eqnarray*} 
  Therefore we obtain
  $$
  \|\mathcal{J}_2\|_{L^\infty}\lesssim \|\rho\|_{W^{2,\infty}}\|\varphi\|_{W^{2,\infty}}.
  $$
  Putting together the preceding estimates allows to get the uniform estimate
  $$
  \sup_{n\in \mathbb{N}^\star} \| (1-\rho) j_0^n\|_{C^{1-\frac2p}}<\infty.
  $$
             It remains to check the assumption \eqref{separa} uniformly with respect to $n.$ This condition should  be a little bit clarified since the singular support of $(\omega_0^n, j_0^n)$ is smoothed out. We replace  in  \eqref{separa} the set $\Sigma_{sing}^{1-\frac2p}(\omega_0^n, j_0^n)$ by $\tilde\Sigma_{sing}^{1-\frac2p}$ defined as follows: we say that $x\notin \tilde\Sigma_{sing}^{1-\frac2p}$ if and only if there exists a smooth compactly supported  function $\chi$ with $\chi(x_0)=1$ such that 
        $\chi \omega_0^n$ and $\chi j_0^n$ belong to $C^{1-\frac2p}$ uniformly \mbox{on $n.$} Performing straightforward calculations one can prove  that 
        $$
        \tilde\Sigma_{sing}^{1-\frac2p}=  \Sigma_{sing}^{1-\frac2p}(\omega_0, j_0).
        $$
    Now since $\{\varphi_n\} $ converges uniformly to $\varphi$   we  can easily  see that  the assumption \eqref{separa}  is satisfied for sufficiently large values of   $n.$
    This achieves the fact that the family $\{v_0^n, b_0^n\}$ is smooth and satisfies the assumptions $(1)-(2)-(3)$ of Theorem \ref{thm2} uniformly with respect to $n.$
    \vspace{0,2cm}
    
    $\bullet$ {\it Uniqueness part}.
 \vspace{0,3cm}
    
    Let $\{(v_i, b_i, p_i), i=1,2\}$ be  two solutions of the system \eqref{MHD} with the same initial data $(v_0, b_0)$ and belonging to  the space $L^\infty_T W^{1,\infty}$ such that $(\omega_i, j_i)\in L^\infty _T(L^1\cap L^\infty)$. We set $v\triangleq v_1-v_2,b\triangleq b_1-b_2$ and $p=p_1-p_2.$ It is known that in general the velocity does not belong to $L^2$ when its vorticity is only bounded and integrable but belongs to $L^p,\,\forall p>2.$ However by reproducing the arguments developed in \cite{che1} we can show the existence of two vector fields $\sigma_1$ and $\sigma_2$ solutions of  the stationary Euler equations and satisfying in addition $\sigma_i\in C^\infty_b$ and $\nabla \sigma_i \in H^s,\forall s\in \RR,$ such that the solutions $v_i$ and $b_i$ constructed in the previous step  belong to $\sigma_1+L^2,\,\sigma_2+L^2$, respectively.   Therefore and in order to give a simple proof for the uniqueness part  we shall assume that $\sigma_i\equiv 0.$  
    
     It is easy to check that $(v,b)$ satisfies the following equations
     \begin{equation}
\label{E0Z}
 \left\{ 
\begin{array}{ll} 
\partial_t v+v_1\cdot\nabla v +\nabla p=b_1\cdot\nabla b-v\cdot\nabla v_2+b\cdot\nabla b_2\\
\partial_t b+v_1\cdot\nabla b=b_1\cdot\nabla v-v\cdot\nabla b_2+b\cdot\nabla v_2.
\end{array} \right.    
     \end{equation} 
     Taking the $L^2-$ inner product of the first equation of \eqref{E0Z} with $v$ and of the second equation with $b$ we find after using the incompressibility of the involved vector fields,
     
     $$
     \frac12\frac{d}{dt}\big(\|v(t)\|_{L^2}^2+\|b(t)\|_{L^2}^2\big)=\int_{\RR^2}\big\{(b_1\cdot\nabla b)\cdot v+(b_1\cdot\nabla v)\cdot b\big\} dx+I(t)
     $$
   with
   $$
   I(t)=\int_{\RR^2}\big(-v\cdot\nabla v_2+b\cdot\nabla b_2\big)\cdot v\,dx+\int_{\RR^2}(-v\cdot\nabla b_2+b\cdot\nabla v_2\big)\cdot b\,dx.
   $$
   Integration by parts shows that the first term of the right-hand side vanishes. For the term $I(t)$ one obtains by using successively  H\"{o}lder and Young  inequalities
   $$
  | I(t)|\lesssim  \big(\|v(t)\|_{L^2}^2+\|b(t)\|_{L^2}^2\big)\big(\|\nabla v_2(t)\|_{L^\infty}+\|\nabla b_2(t)\|_{L^\infty}\big).
   $$
  Consequently
   $$
   \frac{d}{dt}\big(\|v(t)\|_{L^2}^2+\|b(t)\|_{L^2}^2\big)\lesssim \big(\|v(t)\|_{L^2}^2+\|b(t)\|_{L^2}^2\big)\big(\|\nabla v_2(t)\|_{L^\infty}+\|\nabla b_2(t)\|_{L^\infty}\big)
   $$
 and thus  the uniqueness  follows  from Gronwall inequality. 
     \end{proof}
     \section{Commutator estimates} 
We shall in this section discuss some commutator estimates  that   most of them  were of great use in the previous sections.  The first one is technical and 
 whose proof can be found for example in \cite{HKR1}.
\begin{lem}\label{CE} Let $(a, b)\in[1, \infty]^2$ such that $a\ge b'$ with $\frac{1}{b}+\frac{1}{b'}=1$. Given $f, g$ and $h$ three functions such that $\nabla f\in L^{a}$, $g\in L^{b}$ and $xh\in L^{b'}$. Then,
$$
\Vert h\star(fg)-f(h\star g)\Vert_{L^a}\lesssim\Vert xh\Vert_{L^{b'}}\Vert\nabla f\Vert_{L^{a}}\Vert g\Vert_{L^{b}}.
$$
\end{lem}

Next we intend to recall and precise  a crowd of estimates for some
commutators of  Calder\'on type. First, we denote  by $\mathcal{R}_{ij}$ the iterated Riesz transform
$$
\mathcal{R}_{ij}\triangleq\partial_{i}\partial_j\Delta^{-1}.
$$
This operator acts continuously over Lebesgue spaces $L^p$ for $1<p<\infty$ and has an even kernel which is smooth in  $\RR^2\backslash\{0\}$ and with zero  mean value  on the the unit circle. \begin{lem}\label{cald}
Let $f, g:\RR^2\to \RR$ be two smooth functions. Then the following assertions hold true.
\begin{enumerate}
\item  For $p\in]1,\infty[$ we have
$$
\big{\|}[\mathcal{R}_{ij}, f]\partial_kg\big{\|}_{L^p}\lesssim \|\nabla f\big{\|}_{L^\infty}\|g\|_{L^p}.
$$
\item For $\EE\in]0,1[$ and  $p\geq \frac{2}{1-\EE}$, we get
$$
\big\Vert\big[\mathcal{R}_{ij}, f \big]g\big\Vert_{C^\EE}\lesssim\big(\|f\|_{L^\infty}+\|\nabla f\|_{L^\infty}\big)\Vert g\Vert_{ L^p}.
$$
\end{enumerate}
\end{lem}
\begin{proof}

{\bf{$(1)$}} This result follows  from Theorem 1 of  \cite{Cald22}, but in this theorem the dependence of the constant with respect to the norm of $f$ is not precised. However we can obtain our estimate from Theorem 2 of the same paper \cite{Cald22} and we shall outline  in the next lines how to reduce our problem to this case. Let $K_{i}$ denote the Kernel of Riesz transform $\mathcal{R}_{i}\triangleq\partial_i\sqrt{-\Delta}$ which is odd, homogeneous of order $-d$ and belongs to $C^\infty(\RR^d\backslash\{0\})$. Now it is easy to check that 
$$
[\mathcal{R}_{i}, f]\partial_kg(x)=\int_{\RR^2} K_{i}(x-y)\big( f(y)-f(x)\big)\partial_{y_k} g(y) dy.
$$
Thus using integration by parts yields
\begin{eqnarray*}
[\mathcal{R}_{i}, f]\partial_kg(x)&=&-\mathcal{R}_{i}(g\partial_{x_k} f)+\int_{\RR^2} (\partial_{y_k}K_{i})(x-y)\big( f(y)-f(x)\big) g(y) dy\\
&\triangleq&\hbox{I}+\hbox{II}.
\end{eqnarray*}
To estimate the first term we use the continuity of $\mathcal{R}_i: L^p\to L^p$  for $p\in]1,\infty[$ and therefore
\begin{eqnarray*}
\|\hbox{I}\|_{L^p} &\le& C\|g\partial_{x_k} f\|_{L^p}\\
&\le& C\|\nabla f\|_{L^\infty}\|g\|_{L^p}.
\end{eqnarray*}
As regards the second term we shall use Theorem 2 of \cite{Cald22} which is valid in our context since the \mbox{map $x\mapsto \partial_{x_k}K_{i}(x)$} is even,  homogeneous of degree $-d-1$, locally integrable in $\RR^d\backslash\{0\}.$ Consequently
$$
\|\hbox{II}\|_{L^p}\le C\|\nabla f\|_{L^\infty}\|g\|_{L^p}.
$$
Putting together the previous estimates  gives the following result
\begin{equation}\label{riz1}
\big{\|}[\mathcal{R}_{i}, f]\partial_kg\big{\|}_{L^p}\lesssim \|\nabla f\big{\|}_{L^\infty}\|g\|_{L^p}.
\end{equation}
Now let us come back to the iterative Riesz transform $\mathcal{R}_{ij}=\mathcal{R}_i\mathcal{R}_j$ and write

\begin{eqnarray*}
\big[\mathcal{R}_{ij}, f\big]\partial_kg=\mathcal{R}_j\big\{ [\mathcal{R}_{i}, f\big]\partial_kg \big\}+[\mathcal{R}_{j}, f\big]\partial_k\mathcal{R}_ig.
\end{eqnarray*}

Using the preceding result \eqref{riz1}  combined with the continuity of Riesz transforms on $L^p$ lead to the desired result.
\vspace{0,3cm}

{\bf{$(2)$}} We shall use the para-differential calculus through  Bony's decomposition,
 \begin{eqnarray*}
\big[\mathcal{R}_{ij}, f\big]g&=&\sum_{q\in\NN}\big[\mathcal{R}_{ij}, S_{q-1}f\big]\Delta_qg+\sum_{q\in\NN}\big[\mathcal{R}_{ij}, \Delta_{q}f\big]S_{q-1}g
+\sum_{q\ge-1}\big[\mathcal{R}_{ij}, \Delta_{q}f\big]\tilde{\Delta}_q g\\
&\triangleq&\sum_{q\in\NN}\pi_{1}^{q}+\sum_{q\in\NN}\pi_{2}^{q}+\sum_{q\ge-1}\pi_{3}^{q}\\
&\triangleq&\pi_1+\pi_2+\pi_3.
\end{eqnarray*}
To estimate the first term $\pi_1^q$ we use its convolution structure,
$$
\pi_{1}^q=h_q\star (S_{q-1}f\Delta_{q}g)-S_{q-1}f\,(h_q\star\Delta_{q}g),
$$
\noindent where $\widehat{h}_q(\xi)=\frac{\xi_i\xi_j}{|\xi|^2}\psi(2^{-q}\xi)$ and $ \psi$ is a smooth function supported in an annulus centered at  zero. Therefore $h_q(x)=2^{2q}h(2^{q}x)$ with $h\in\mathcal{S}$. Then in view of the Lemma \ref{CE} we get,
\begin{equation*}
\Vert \pi_{1}^{q}\Vert_{L^{\infty}}\lesssim\Vert xh_q\Vert_{L^{1}}\Vert\nabla S_{q-1}f\Vert_{L^{\infty}}\Vert\Delta_{q}g\Vert_{L^\infty}. 
\end{equation*}
Using the fact $\Vert xh_q\Vert_{L^{1}}=2^{-q}\Vert xh\Vert_{L^{1}}$ combined with Bernstein inequality we obtain with the assumption $p\geq\frac{2}{1-\EE}$
\begin{eqnarray*}
2^{q\EE}\Vert \pi_{1}^q\Vert_{L^{\infty}}
&\lesssim&2^{q(-1+\frac{2}{p}+\EE)}\Vert\Delta_{q}g\Vert_{L^p}\|\nabla f\|_{L^\infty}\\
&\lesssim& \Vert g\Vert_{L^p}\|\nabla f\|_{L^\infty}.
\end{eqnarray*}
Since
$$
\Delta_j\sum_{q\in\NN}\pi_{1}^q=\sum_{\vert j-q\vert\le4}\pi_{1}^{q}
$$ 
then it follows
\begin{eqnarray}\label{Ay05}
\Vert\pi_1\Vert_{C^{\varepsilon}}&\lesssim&  \Vert g\Vert_{L^p}\|\nabla f\|_{L^\infty}.
\end{eqnarray}
Concerning the second term $\pi_{2}^{q}$, we  follow the same steps of the preceding case
\begin{eqnarray*}
2^{q\EE}\Vert\pi_{2}^q\Vert_{L^{\infty}}
&\lesssim&2^{q(-1+\EE)}\Vert S_{q-1}g\Vert_{L^\infty}\|\nabla \Delta_qf\|_{L^\infty}\\
&\lesssim&2^{q(-1+\EE+\frac2p)}\Vert S_{q-1}g\Vert_{L^p}\|\nabla f\|_{L^\infty}\\
&\lesssim& \Vert g\Vert_{L^p}\|\nabla f\|_{L^\infty}.
\end{eqnarray*}

Now we can conclude in a similar way to the first term $\pi_1$ that
\begin{equation}\label{Ay06}
\Vert\pi_2\Vert_{C^{\varepsilon}}\lesssim \Vert g\Vert_{L^p}\|\nabla f\|_{L^\infty}.
\end{equation}
Let us now move to the third term $\pi_3$. By the definition of the remainder term we have
\begin{eqnarray*} 
\Vert\Delta_q\pi_3\Vert_{L^\infty}&\lesssim&\sum_{k\geq q-3}\Vert\big[\mathcal{R}_{ij}, \Delta_{k}f\big]\tilde{\Delta}_{k}g\Vert_{L^\infty}+\big\|\big[\mathcal{R}_{ij}, \Delta_{-1}f\big]\tilde{\Delta}_{-1}g\big\|_{L^\infty}\chi_{[-1,4]}(q)\\
&\lesssim&\Big\{\sum_{k\geq q-3}\big\|\mathcal{R}_{ij}(\Delta_{k}f\,\tilde{\Delta}_{k}g)\big\|_{L^\infty}+\sum_{k\geq q-3}\big\|\Delta_{k}f(\mathcal{R}_{ij}\tilde{\Delta}_{k}g)\big\|_{L^\infty}\Big\}\\&+&\big\|\big[\mathcal{R}_{ij}, \Delta_{-1}f\big]\tilde{\Delta}_{-1}g\big\|_{L^\infty}\chi_{[-1,4]}(q)\\
&\triangleq&\mathcal{I}_q+\mathcal{II}_q.
\end{eqnarray*}
By Bernstein inequality and the continuity of Riesz transforms over $L^p$ we get
\begin{eqnarray}\label{Ay03}
\nonumber 2^{q\EE}\mathcal{I}_q&\lesssim&\|g\|_{L^p}2^{q\EE}\sum_{k\geq q-3}2^{k\frac2p}\|\Delta_kf\|_{L^\infty}\\
\nonumber &\lesssim&\|g\|_{L^p}\|\nabla f\|_{L^\infty}2^{q\EE}\sum_{k\geq q-3}2^{k(\frac2p-1)}\\
&\lesssim&\|g\|_{L^p}\|\nabla f\|_{L^\infty}.
\end{eqnarray}
For the low frequency term $\mathcal{II}_q$ we use once again Bernstein inequality combined with the continuity of Riesz transforms over $L^p$
\begin{eqnarray*}
\|\mathcal{II}_q\|_{L^p}&\lesssim& \|\Delta_{-1}f\|_{L^\infty}\|\tilde\Delta_{-1}g\|_{L^p}\\
&\le&  \|f\|_{L^\infty}\|g\|_{L^p}.
\end{eqnarray*}

Consequently we find
\begin{equation}\label{Ay08}
\|\pi_{3}\|_{C^\EE}\lesssim  \big(\|f\|_{L^\infty}+\|\nabla f\|_{L^\infty}\big)\|g\|_{L^p}.	
\end{equation}
Therefore putting together  (\ref{Ay05}), (\ref{Ay06}) and (\ref{Ay08}) yields 
\begin{equation*}
\big\Vert\big[\mathcal{R}_{ij},f\big]g\big\Vert_{C^{\varepsilon}}\lesssim(\|f\|_{L^\infty}+\|\nabla f\|_{L^\infty}\big)\|g\|_{L^p}.
\end{equation*}
This completes the proof of the commutator estimate. 

\end{proof}
Now, we introduce the following  operator $\mathcal{L}:=\partial_i\Delta^{-1}$ which is of convolution type and our aim is to establish a commutator estimate between this singular operator and the convection \mbox{operator $v\cdot\nabla$. }
 \begin{lem}\label{An1} Let $\varepsilon\in]0, 1[,\,p, \in ]1,\infty[.$ Let $ \rho :\RR^2\to\RR $ be a smooth function  and $v$ be a smooth divergence-free vector field \mbox{on $\RR^2$}.  Then   

$$
\big\Vert\big[\mathcal{L}, v\cdot\nabla \big]\rho\big\Vert_{C^{\EE}}\lesssim\|v\|_{C^{\EE}}\Vert \rho\Vert_{L^\infty\cap L^p}.
$$

\end{lem} 
\begin{proof} The proof will be done in the spirit of the preceding one.  From Bony's decomposition,
\begin{eqnarray*}
\big[\mathcal{L}, v\cdot\nabla\big]\rho&=&\sum_{q\in\NN}\big[\mathcal{L}, S_{q-1}v\cdot\nabla\big]\Delta_q\rho+\sum_{q\in\NN}\big[\mathcal{L}, \Delta_{q}v\cdot\nabla\big]S_{q-1}\rho
+\sum_{q\ge-1}\big[\mathcal{L}, \Delta_{q}v\cdot\nabla\big]\tilde{\Delta}_q\rho\\
&\triangleq&\sum_{q\in\NN}\pi_{1}^{q}+\sum_{q\in\NN}\pi_{2}^{q}+\sum_{q\ge-1}\pi_{3}^{q}\\
&\triangleq&\pi_1+\pi_2+\pi_3.
\end{eqnarray*}
To estimate the first term $\pi_1^q$ we use its convolution structure,
$$
\pi_{1}^q=h_q\star (S_{q-1}v\Delta_{q}\nabla\rho)-S_{q-1}v(h_q\star\nabla\Delta_{q}\rho),
$$
\noindent where $\widehat{h}_q(\xi)=\frac{\xi_i}{|\xi|^2}\psi(2^{-q}\xi)$ and $ \psi$ is a smooth function supported in an annulus with center zero. Therefore $h_q(x)=2^{q}h(2^{q}x)$ with $h\in\mathcal{S}$ and  in view of Lemma \ref{CE} we get,
\begin{eqnarray*}
\Vert \pi_{1}^{q}\Vert_{L^{\infty}}&\lesssim& 2^{-2q}\Vert\nabla S_{q-1}v\Vert_{L^{\infty}}\Vert\Delta_{q}\nabla\rho\Vert_{L^\infty}\\
&\lesssim& 2^{-q}\Vert\nabla S_{q-1}v\Vert_{L^{\infty}}\Vert\Delta_{q}\rho\Vert_{L^\infty}.
\end{eqnarray*}
Hence we obtain since $\EE<1$
\begin{eqnarray*}
2^{q\EE}\Vert\pi_{1}^q\Vert_{L^{\infty}}
&\lesssim&2^{q(-1+\EE)}\Vert\rho\Vert_{L^\infty}\sum_{-1\le j\le q-2}\Vert\nabla\Delta_{j}v\Vert_{L^{\infty}}\\
&\lesssim&2^{q(-1+\EE)}\Vert\rho\Vert_{L^\infty}\sum_{-1\le j\le q-2}2^{j(1-\EE)}\Vert v\Vert_{C^{\EE}}\\
&\lesssim&\Vert\rho\Vert_{L^\infty} \Vert v\Vert_{C^{\EE}}.
\end{eqnarray*}
Therefore we get
\begin{eqnarray}\label{Ay5}
\Vert\pi_1\Vert_{C^{\varepsilon}}&\lesssim& \Vert\rho\Vert_{L^\infty} \Vert v\Vert_{C^{\EE}}.
\end{eqnarray}
Concerning the second term $\pi_{2}^{q}$, we imitate the previous computations
\begin{eqnarray*}
\Vert \pi_{2}^{q}\Vert_{L^{\infty}}&\lesssim& 2^{-2q}\Vert\nabla \Delta_qv\Vert_{L^{\infty}}\Vert S_{q-1}\nabla\rho\Vert_{L^\infty}\\
&\lesssim& \Vert \Delta_q v\Vert_{L^{\infty}}\Vert \rho\Vert_{L^\infty}.
\end{eqnarray*}
Therefore we obtain 
\begin{eqnarray*}
2^{q\EE}\Vert \pi_{2}^{q}\Vert_{L^{\infty}}&\lesssim&
\Vert v\Vert_{C^\EE}\Vert\rho\Vert_{L^\infty}
\end{eqnarray*}
and consequently
\begin{eqnarray}\label{Ay6}
\Vert\pi_2\Vert_{C^{\varepsilon}}&\lesssim& \Vert\rho\Vert_{L^\infty} \Vert v\Vert_{C^{\EE}}.
\end{eqnarray}
Let us now move to the third term $\pi_3$. By the definition of the remainder term we have
\begin{eqnarray*} 
\pi_3
&=&\sum_{q\geq-1}\mathcal{L}\Div(\Delta_{q}v\,\tilde{\Delta}_{q}\rho)-\sum_{q\geq-1}\Delta_{q}v\cdot\nabla\mathcal{L}(\tilde{\Delta}_{q}\rho)\\
&\triangleq&\pi_3^1-\pi_3^2.
\end{eqnarray*}
By Bernstein inequality we obtain for $j\in \NN$
\begin{eqnarray*}
2^{j\varepsilon}\Vert\Delta_j\pi_3^1\Vert_{L^{\infty}}&\lesssim&2^{j\varepsilon}\sum_{q\geq j-4}\|\Delta_q v\|_{L^\infty}\|\tilde{\Delta}_{q}\rho\|_{L^\infty}\\
\nonumber&\lesssim&\|\rho\Vert_{L^{\infty}}\|v\|_{C^\varepsilon}\sum_{q\geq j-4}2^{(j-q)\varepsilon}\\
\nonumber&\lesssim&\Vert\rho\Vert_{L^{\infty}}\Vert v\Vert_{C^{\EE}}.
\end{eqnarray*}
For the low frequency we use the continuity of Riesz transforms over $L^p$
\begin{eqnarray*}
\Vert\Delta_{-1}\pi_3^1\Vert_{L^{\infty}}&\lesssim&\sum_{q\geq -1}\|\Delta_q v\|_{L^\infty}\|\tilde{\Delta}_{q}\rho\|_{L^p}\\
\nonumber&\lesssim&\Vert v\Vert_{C^{\EE}}\Vert\rho\Vert_{L^{p}}.
\end{eqnarray*}
Thus we find
\begin{eqnarray}\label{Ay3}
\Vert\pi_3^1\Vert_{C^\varepsilon}&\lesssim&\Vert v\Vert_{C^{\EE}}\Vert\rho\Vert_{L^p\cap L^\infty}.
\end{eqnarray}
As regards the  term $\pi_3^2$ we write
\begin{eqnarray*}
\pi_3^2&=&\sum_{q\geq2}\Delta_{q}v\cdot\nabla\mathcal{L}(\tilde{\Delta}_{q}\rho)+\sum_{q=-1}^1\Delta_{q}v\cdot\nabla\mathcal{L}(\tilde{\Delta}_{q}\rho)\\
&\triangleq&\pi_3^{2,1}+\pi_3^{2,2}.
\end{eqnarray*}
Since for $q\geq2$ the Fourier transform of $\tilde{\Delta}_{q}\rho$ is supported in an annulus of size $2^q$ then
\begin{eqnarray*}
\nonumber2^{j\varepsilon}\Vert\Delta_j\pi_3^{2,1}\Vert_{L^\infty}&\lesssim&2^{j\varepsilon}\sum_{q\geq j-4; q\geq2}\Vert\Delta_{q}v\Vert_{L^{\infty}}\Vert\nabla \mathcal{L}(\tilde{\Delta}_{q}\rho)\Vert_{L^\infty}\\
\nonumber&\lesssim& \|\rho\|_{L^\infty}2^{j\varepsilon}\sum_{q\geq j-4}\Vert\Delta_{q}v\Vert_{L^{\infty}}\\
&\lesssim&\|\rho\|_{L^\infty}\|v\|_{C^\varepsilon}
\end{eqnarray*}
For the term $\pi_3^{2,2}$ we get
\begin{eqnarray*}
\|\pi_3^{2,2}\|_{C^\varepsilon}&\lesssim& \sum_{q=-1}^1\|\Delta_{q}v\|_{L^\infty}\|\nabla\mathcal{L}(\tilde{\Delta}_{q}\rho)\|_{L^p}\\
&\lesssim&\|v\|_{C^\varepsilon}\|\rho\|_{L^p}.
\end{eqnarray*}
It follows that
\begin{eqnarray}\label{Ay4}
\Vert\pi_3^2\Vert_{C^\varepsilon}&\lesssim&\Vert v\Vert_{C^{\EE}}\Vert\rho\Vert_{L^{p}\cap L^\infty}.
\end{eqnarray}

Putting together (\ref{Ay3}) and (\ref{Ay4}) we find
\begin{eqnarray*}
\Vert \pi_3\Vert_{C^{\varepsilon}}&\lesssim&\|v\|_{C^\varepsilon}\|\rho\|_{L^p\cap L^\infty}.
\end{eqnarray*}
Combining this estimate with (\ref{Ay5}) and (\ref{Ay6})  yields for any $p\in]1,\infty[$and  $0<\EE<1,$
\begin{equation}\label{Ya1}
\big\Vert\big[\mathcal{L}, v\cdot\nabla\big]\rho\big\Vert_{C^{\varepsilon}}\lesssim\Vert v\Vert_{C^{\EE}}\Vert\rho\Vert_{L^{p}\cap L^\infty}.
\end{equation}
This completes the proof of the commutator estimate. 
\end{proof}

\begin{gracies}
The author would like to thank   Joan Verdera  for  the fruitful  discussion about the stationary patches and for  pointing out the valuable references \cite{Fran} and \cite{Rei}. Special thanks also go to Camil Muscalu who gave  me more insights and clarifications about  Calder\'on commutators.   \end{gracies}

\end{document}